\documentclass[a4paper,12pt]{article} 

\usepackage{graphicx}

\usepackage{amssymb,latexsym}
\usepackage{bm, amsmath, amssymb, amsthm}
\usepackage[right=2.5cm, left=2.5cm, top=2.5cm, bottom=2.5cm]{geometry}

\newtheorem{theorem}{Theorem}[section]

\newtheorem{corollary}[theorem]{Corollary} 
\newtheorem{lemma}[theorem]{Lemma}

\theoremstyle{definition}
\newtheorem{definition}[theorem]{Definition}
\newtheorem{example}[theorem]{Example}
\newtheorem{remark}[theorem]{Remark}

\newcommand\<{\langle}
\renewcommand\>{\rangle}

\newcommand\tr{\operatorname{Tr\,}}

\newcommand\Div{\operatorname{div}}

\newcommand\mD{{\cal D}}

\newcommand\mP{{\cal P}}
\newcommand\mF{{\cal F}}
\newcommand\mK{{\cal K}}
\newcommand\mH{{\cal H}}

\newcommand\mV{{\cal V}}
\newcommand\mS{{\cal S}}
\newcommand\mN{{\cal N}}
\newcommand\mT{{\cal T}}

\newcommand\vol{\operatorname{vol}}

\newcommand\Riem{\operatorname{Riem}}

\newcommand\dt{\partial_t}
\newcommand\ddt{\partial^2_{tt}}

\newcommand\tT{{\tilde T}^{\sharp t}}

\newcommand\tA{{\tilde A}^t}
\newcommand\ee{{\cal E}}

\begin{document}

\title{Variations of metric that preserve a Riemannian submersion and geometry of its fibers}

\author{ Tomasz Zawadzki\footnote{Faculty of Mathematics and Computer Science, University of \L\'{o}d\'{z}, ul. Banacha 22, 90-238 \L\'{o}d\'{z}, Poland;
e-mail: {\tt tomasz.zawadzki@wmii.uni.lodz.pl}  }}

\date{}

\maketitle 

\begin{abstract}
On the domain of a Riemannian submersion, we consider variations (i.e., smooth one-parameter families) of Riemannian metrics, for which the submersion is Riemannian and which all keep the metric induced on its fibers fixed. We obtain a formula for the variation of the second fundamental form of the fibers with respect to such changes of metric. We find a choice of parameters defining the variations, that allows to easily formulate the necessary and sufficient conditions for preserving particular geometry of the fibers, i.e., keeping them totally geodesic, totally umbilical, or minimal. These conditions are related to the existence of Killing, conformal Killing and divergence-free vector fields on the fibers. 
We find conditions for metric to be a critical point of integrated squared norms of the mean curvature and the second fundamental form of the fibers, with respect to the considered variations, and prove that at all critical points of these functionals the second variation is non-negative. We also examine variations of sectional curvatures of planes defined by the horizontal lifts of vectors from the image of the submersion. In particular, we find that some variations preserving Riemannian submersions with totally geodesic fibers can make vertizontal curvatures not constant on the fibers.

\noindent
\textbf{Keywords}: Riemannian submersion, distribution, variation, second fundamental form, mean curvature

\noindent
\textbf{Mathematics Subject Classifications (2020)} 
53C12; 
53C15
\end{abstract}

\section{Introduction}

Riemannian submersions are a classical tool in differential geometry \cite{Gray,O'Neill}, used e.g., to produce examples of manifolds satisfying particular curvature conditions \cite{Cheeger, Ziller}. 
On the domain $(M,g)$ of a Riemannian submersion $\pi : (M, g) \rightarrow (N, g_N)$ we have two orthogonal distributions: the vertical one, tangent to the fibers of $\pi$, and the horizontal one, orthogonal to the fibers and mapped by $\pi_*$ isometrically onto the tangent bundle $TN$. 
The metric on the horizontal distribution is thus determined by 
$\pi$ and $g_N$, which leads to two complementary ways of deforming $g$, while keeping $\pi$ a Riemannian submersion. One is to change $g$ only along the fibers of $\pi$, preserving both the metric on the horizontal distribution and the horizontal distribution itself - this is the case of e.g., canonical variations \cite{Besse}, or conformal changes of metric along the fibers \cite{March, Nagy}. 
Another way 
is to 
consider metrics fixed along the fibers, but with 
various horizontal distributions isometric (by $\pi_*$) to $TN$. For example, in \cite{AlvarezLopez} the existence of such horizontal distribution for which the mean curvature of fibers is basic was proved.  

In this paper we consider
one-parameter families of metrics $\{ g_t, t \in (-\epsilon , \epsilon) \}$ on the domain $M$ of a Riemannian submersion $\pi : (M,g_0) \rightarrow (N,g_N)$, where the restrictions of metrics to every fiber remain unchanged, but for different values of $t$ a different distribution may be orthogonal to the fibers with respect to $g_t$. As the horizontal distribution changes, the metric on it gets rescaled, so that $\pi_*$ keeps it isometric to $TN$. In effect, $\pi: (M, g_t) \rightarrow (N, g_N)$ remains a Riemannian submersion for all $t \in (-\epsilon , \epsilon)$ and on every fiber of $\pi$ all $g_t$ induce the same metric as $g_0$. While the intrinsic geometry of the fibers remains fixed, we examine how their extrinsic geometry (defined by the second fundamental form) varies, and how it can be preserved. 

The description of one-parameter families (called variations) of metrics, that preserve a given Riemannian submersion, gives a straightforward condition, relating components of $B_t = \dt g_t$ tensor field, given in Theorem \ref{thBXYriem}. Obtaining such
variation amounts to solving a certain system of 
linear ordinary differential equations at every point of $M$, where particular components of $B_t$ in a $g_0$-orthonormal frame are parameters controlling the ODEs and can be set arbitrarily. However, we find it more useful to consider 
an equivalent, non-linear ODE system determining the metric, with 
parameters closely related to properties of $B_t^\sharp$ tensor field, which is defined by $g_t(B_t^\sharp X, Y) = B_t(X,Y)$ for all vector fields $X,Y$ on $M$. These new parameters are 
a set of vertical vector fields, paired with $g_0$-horizontal lifts of vector fields from $N$, defined in Theorem \ref{corbsharpglobal}. In Section \ref{secpresgeom}, we find a simple formula for variation of the second fundamental form of the fibers, in terms of those vector fields, and conditions for keeping the fibers totally geodesic, totally umbilical or minimal - which require the vector fields to be Killing, conformal Killing or divergence-free (respectively), when restricted to the fibers. 

In Section \ref{secvarprobs} we examine variational problems for extrinsic geometry of the fibers of a Riemannian submersion $\pi$, considering $L_2$ norms of the second fundamental form and the mean curvature of the fibers as functionals on the space of those Riemannian metrics, that preserve the Riemannian submersion and are fixed along its fibers. These functionals can be also defined for a single fiber; their critical points, found in Section \ref{seccritpoints}, are metrics satisfying simply formulated conditions, respectively: vanishing of a certain divergence-type operator on the second fundamental form, and the mean curvature of the fibers being a projectable vector field. The choice of parameters determining variations allows to easily compute the second variations of the functionals in Section \ref{secsecondvar}, 
which at all critical points are non-negative. 

In Section \ref{secseccurv} we consider how the sectional curvatures in horizontal and vertical directions change with variations of metric preserving a Riemannian submersion.
In Section \ref{seccurvXY} we examine the sectional curvature of the horizontal lift of a local plane field on $N$ and its variations with respect to the considered changes of metric, that preserve the Riemannian submersion, but move the horizontal lift with respect to the fibers. 
This variational problem admits as critical points only metrics 
for which the distribution arising as the horizontal lift of the plane field satisfies a certain integrability condition. In general, variation of such curvature at $x \in M$ depends both on behaviour of vector fields defining the plane on $N$ in a neighbourhood of $\pi(x)$, and on behaviour of vector fields defining the variation in a neighbourhood of $x$. On the other hand, variational problem for integrated horizontal scalar curvature (i.e., the sum of sectional curvatures of planes defined by an orthogonal set of horizontal vectors) on a compact domain of a Riemannian submersion is well posed and reduces to a problem for the norm of the integrability tensor of the horizontal distribution, examined e.g., in \cite{rz-2,TZCMS}. 

In Section \ref{seccurvXU} we consider variations of the vertizontal curvatures, i.e., sectional curvatures of planes spanned by a horizontal and a vertical vector. Assuming existence of certain Killing fields on the fibers, in Corollary \ref{corpressecUXex} we demonstrate existence of Riemannian submersions with totally geodesic fibers and vertizontal curvatures non-constant on a fiber. This construction can be applied to modify e.g., the Hopf fibrations by $3$-spheres, producing new fat Riemannian submersions \cite{ZillerFatness}.

\section{Notation and definitions} \label{secnotation}

In what follows, $M$ and $N$ are smooth, connected and oriented manifolds, and all functions and tensor fields on them are always assumed to be smooth. For a manifold $M$, $\mathfrak{X}_M$ denotes the set of vector fields on $M$; for a distribution ${\mD}$ on $M$ (i.e., a subbundle of the tangent bundle $TM$) the set of vector fields with values in ${\mD}$ will be denoted by $\mathfrak{X}_{{\mD}}$. By $\mD_x$ we denote the distribution $\mD$ at the point $x \in M$, i.e., the subspace $\mD_x \subset T_xM$ spanned by values at $x$ of all vector fields in $\mathfrak{X}_{\mD}$.

For a Riemannian metric $g_0$ on a manifold $M$ with a distribution $\mD$, by $\Riem(M, {\mD}, g_0)$ we denote the set of all Riemannian metrics $g$ on $M$, such that $g(X,Y) = g_0(X,Y)$ for all $X,Y \in \mathfrak{X}_{{\mD}}$. We will consider one-parameter families of metrics $\{ g_t , t \in (-\epsilon, \epsilon) \}$ 
in $\Riem(M, {\mD}, g_0)$, 
which are called variations of $g_0$, and are always assumed to smoothly depend on the parameter $t$. The first and the second derivative with respect to $t$ will be denoted by $\dt$ and $\ddt$, respectively. For every variation $g_t$, let $B_t = \dt g_t$ and let $B_t^\sharp$ be the tensor field defined by $g_t(B_t^\sharp X, Y) = B_t(X,Y)$ for all $X,Y \in \mathfrak{X}_M$. Let $\nabla^t$ be the Levi-Civita connection of $g_t$. We say that a vector field $X \in \mathfrak{X}_M$ is bounded on $(M,g)$ if $g(X,X)$ is a bounded function on $M$.

We recall some information about Riemannian submersions \cite{O'Neill}. 
A Riemannian submersion $\pi : (M,g) \rightarrow (N, g_N)$ is a smooth map between Riemannian manifolds such that rank of $\pi_*$ is equal to $\dim N$ and $\pi_*$ is an isometry on vectors orthogonal to the fibers; more precisely, for all vector fields $X,Y$ on $M$, such that $g(X,\xi) = 0 = g(Y, \xi)$ for all $\xi \in \ker \pi_*$, we have
\[
g(X,Y) = g_N (\pi_* X, \pi_* Y).
\]
The distribution $\mV = \ker \pi_*$ is called vertical, its $g$-ortogonal complement $\mH$ is called the $g$-horizontal distribution; a vector field is called vertical or $g$-horizontal if it has values in the corresponding distribution. We will use notation $n=\dim \mV$ and $n+p = \dim M$. We say that a vector field $X \in \mathfrak{X}_M$ is 
projectable if $\pi_* X$ is well defined; $X$ is 
projectable if and only if $[X, V] \in \mathfrak{X}_{\mV}$ for all $V \in \mathfrak{X}_{\mV}$. For every $W \in \mathfrak{X}_N$ there exists a $g$-horizontal lift of $W$, i.e., a unique $g$-horizontal vector field $\pi^* W$ on $M$ such that $\pi_* (\pi^* W) =W$. For all $x \in N$, let $\mF_x = \pi^{-1}(\{x\})$ be the fiber of $\pi$ over $x$.

Let $\pi : (M, g_0) \rightarrow (N, g_N)$ be a Riemannian submersion.
For a one-parameter family $\{ g_t, t \in (-\epsilon, \epsilon) \} \subset \Riem(M, 
\mV , g_0)$, let $\mH(t)$ denote the $g_t$-horizontal distribution of $\pi$.
Let $P_{\mV}^t$ and $P_{\mH}^t $ denote the $g_t$-orthogonal projections onto the distributions $\mV$ and $\mH(t)$, respectively. We shall consider sets of linearly independent
vector fields $\{ E_1, \ldots, E_n , \ee_{n+1} , \ldots \ee_{n+p} \}$ defined on an open subset of $M$, where $E_a \in \mathfrak{X}_{\mV}$ for all $a \in \{1, \ldots, n\}$; we will call them local adapted frames and denote such a frame shortly by $\{E_a, \ee_i\}$. If vector fields in a frame $\{E_a, \ee_i\}$ are $g$-orthonormal, we will call such frame $g$-orthonormal. For indices we use the convention $a,b,c \in \{1, \ldots, n\}$, $i,j,k \in \{n+1, \ldots, n+p\}$ and $\mu, \nu \in \{1, \ldots, n+p\}$.

The second fundamental form ${\tilde h}_t$ and the integrability tensor ${\tilde T}_t$ of the distribution $\mH(t)$ are defined by formulas:
\[
{\tilde h}_t(X,Y) = \frac{1}{2} P_{\mV}^t (\nabla^t_X Y + \nabla^t_Y X), \quad {\tilde T}_t(X,Y) = \frac{1}{2} P_{\mV}^t [X,Y]
\]
for all $X,Y \in \mathfrak{X}_{\mH(t)}$. We also define tensor fields $\tA$ and $\tT$ by formulas
\[
g_t(\tA_{Z} X, Y) = g_t({\tilde h}_t(X,Y) ,Z) , \quad g_t(\tT_{Z} X, Y) = g_t({\tilde T}_t(X,Y) ,Z)
\]
for all $X,Y \in \mathfrak{X}_{\mH(t)}$ and all $Z \in \mV$. If $\pi : (M, g_t) \rightarrow (N, g_N)$ is a Riemannian submersion, we have ${\tilde h}_t =0$ and $\tA =0$ \cite{O'Neill}. 
The second fundamental form of the fibers of $\pi : (M, g_t) \rightarrow (N, g_N)$ is defined as
\[
h_t(X,Y) = P_{\mH}^t \nabla_{X} Y  
\]
for all $X,Y \in \mathfrak{X}_{\mV}$; we have $h_t(X,Y)=h_t(Y,X)$.
We say that the fibers of $\pi : (M, g_t) \rightarrow (N, g_N)$ 
are totally geodesic if for all $X,Y \in \mathfrak{X}_{\mV}$ we have 
$h_t(X,Y)=0$.
Let 
$\{E_1 ,  \ldots , E_n\}$ 
be a local $g_t$-orthonormal frame spanning $\mV$, we define the mean curvature of the fibers as 
\[
H_t = \sum\nolimits_{a=1}^n h_t(E_a, E_a). 
\]
We say that $\pi : (M, g_t) \rightarrow (N, g_N)$ has minimal fibers if $H_t=0$. We say that $\pi : (M, g_t) \rightarrow (N, g_N)$ has totally umbilical fibers if for all $X,Y \in \mathfrak{X}_{\mV}$ we have $h_t(X,Y) = \frac{1}{n} g_t(X,Y) H_t$. 

Recall that for $x \in N$, $\mF_x = \pi^{-1}(\{x\})$. 
We define the vertical divergence of a vector field $X \in \mathfrak{X}_{M}$ by 
\[
\Div^t_{\mV} X = \sum\nolimits_{a=1}^n g_t(\nabla^t_{E_a} X, E_a),
\]
where 
$\{E_1 ,  \ldots , E_n\}$  
are local $g_t$-orthonormal vertical vector fields. Let $X \in \mathfrak{X}_{\mF_x}$, then $\Div^t_{\mV} X$ defined by the above formula is the divergence of $X$ on the Riemannian manifold $(\mF_x, g_t \vert_{\mF_x})$ and if $\Div^t_{\mV} X =0$, we will say that such $X$ is divergence-free on $(\mF_x, g_t \vert_{\mF_x})$.

We will also use another divergence-type operator, $\delta_{t} : \mS_{r,2}(\mF_x) \rightarrow \mT_{r,1}(\mF_x)$, acting from the space $\mS_{r,2}(\mF_x)$ of symmetric $(r,2)$-tensor fields on $\mF_x$ to the space $\mT_{r,1}(\mF_x)$ of $(r,1)$-tensor fields on $\mF_x$,
defined by the formula 
\[
(\delta_{t} S)(X) = \sum\nolimits_{a=1}^n ( \nabla^t_{E_a} S)(E_a, X) 
\] 
for all $X \in \mathfrak{X}_{\mF_x}$ and $S \in \mS_{r,2}(\mF_x)$. We note that for $r=0$, $\mT_{0,1}(\mF_x)=\Omega^1(\mF_x)$ is the space of $1$-forms on $\mF_x$ and the formal adjoint of $\delta_{t}$ is given by $( \delta^*_t \omega)(X,Y) = \frac{1}{2} (\mathcal{L}_{\omega^\sharp} g_t)(X,Y)$, for all $\omega \in \Omega^1(\mF_x)$ and all $X,Y \in \mathfrak{X}_{\mF_x}$, where $\mathcal{L}_{\omega^\sharp} g_t$ is the Lie derivative of $g_t$ with respect to vector field $\omega^\sharp$, which is defined by $g_t( \omega^\sharp , Z)=\omega(Z)$ for all $Z \in \mathfrak{X}_{\mF_x}$ \cite{Besse}.

\section{Variations of metric}

\subsection{Preserving Riemannian submersions}

\begin{definition}
Let $\pi : (M, g_0) \rightarrow (N, g_N)$ be a Riemannian submersion. 
We say that a 
variation $\{ g_t, t \in (-\epsilon, \epsilon) \} \subset \Riem(M, 
\mV , g_0)$ preserves the Riemannian submersion $\pi$, if $\pi : (M, g_t) \rightarrow (N, g_N)$ is a Riemannian submersion for all $t \in (-\epsilon, \epsilon)$.
\end{definition}

\begin{theorem} \label{thBXYriem}
Let  $\pi : (M, g_0) \rightarrow (N, g_N)$ be a Riemannian submersion. 
A 
variation $\{ g_t, t \in (-\epsilon, \epsilon) \} \subset \Riem(M, \mV , g_0)$ preserves the Riemannian submersion $\pi$
if and only if for all $t \in (-\epsilon, \epsilon)$
\begin{equation} \label{BXYfields}
B_t(X,Y) = B_t(P_{\mH}^t X , P_{\mV}^t Y ) + B_t(P_{\mH}^t Y, P_{\mV}^t X) \quad {\rm for \; all } \; X,Y \in \mathfrak{X}_{M} .
\end{equation}
\end{theorem}
\begin{proof}
Using $\{ g_t, t \in (-\epsilon, \epsilon) \} \subset \Riem(M, \mV , g_0)$, we obtain 
 \begin{equation} \label{BVVzero}
B_t(U,V) = 0 \quad {\rm for \; all } \; U,V \in \mathfrak{X}_{\mV}
\end{equation} 
from $B_t(U,V) = \dt g_t(U,V) = \dt g_0(U,V)=0$. 

Using $P_{\mV}^t Z + P_{\mH}^t Z = Z$ and $\pi_* P_{\mV}^t  Z =0$ for all $Z \in \mathfrak{X}_M$ and all $t \in (-\epsilon, \epsilon)$, we obtain that 
$\pi : (M, g_t) \rightarrow (N, g_N)$ is a Riemannian submersion if and only if for all $X,Y \in \mathfrak{X}_M$ we have
\begin{eqnarray*}
g_t(P_{\mH}^t X, P_{\mH}^t Y) &=& g_N (\pi_* P_{\mH}^t  X, \pi_*  P_{\mH}^t  Y ) = g_N (\pi_* X, \pi_* Y ) \\ 
&=& g_N(\pi_* P_{\mH}^0  X, \pi_* P_{\mH}^0  Y) = g_0(P_{\mH}^0 X, P_{\mH}^0 Y). 
\end{eqnarray*}
The above holds if and only if for all $x \in M$, $X,Y \in T_xM$ and every $g_0$-orthonormal -- and hence $g_t$-orthonormal for all $t \in  (-\epsilon , \epsilon)$, because $\{ g_t, t \in (-\epsilon, \epsilon) \} \subset \Riem(M, \mV , g_0)$ -- basis 
$\{E_1 ,  \ldots , E_n\}$ of ${\mV}_x$ we have
\begin{eqnarray} \label{BXY}
0 &=&  \dt g_t( P_{\mH}^t  X, P_{\mH}^t Y ) = \dt \big( g_t (X - \sum\nolimits_{a=1}^n g_t(X, E_a)E_a \, , Y- \sum\nolimits_{a=1}^n g_t(Y, E_b)E_b  ) \big) \nonumber \\
&=& B_t(X,Y) - \dt \sum\nolimits_{a=1}^n g_t(X,E_a) g_t (E_a,Y) - \dt \sum\nolimits_{a=1}^n g_t(Y, E_b) g_t (E_b, X) \nonumber \\
&& + \dt \sum\nolimits_{a,b=1}^n g_t(X, E_a) g_t(E_b , Y) g_t(E_a, E_b)  \nonumber \\
&=& B_t(X,Y) - \sum\nolimits_{a=1}^n B_t(X,E_a) g_t(Y,E_a) - \sum\nolimits_{a=1}^n B_t(Y,E_a) g_t(X,E_a),
\end{eqnarray}
where we used $g_t(E_a, E_b) = g_0(E_a, E_b) = \delta_{ab}$. 
From \eqref{BXY}, $P_{\mV}^t Z = \sum\nolimits_{a=1}^n g_t(Z, E_a)E_a$, $Z = P_{\mV}^t Z + P_{\mH}^t Z$ for all $Z \in \mathfrak{X}_M$ and \eqref{BVVzero}, we obtain \eqref{BXYfields}.
\end{proof} 

In particular, from \eqref{BXYfields} follows that every 
variation $\{ g_t, t \in (-\epsilon, \epsilon) \} \subset \Riem(M, \mV , g_0)$, 
that preserves the Riemannian submersion $\pi$, 
is determined by values of $B_t(Z, V)$ for all $Z \in \mathfrak{X}_M$ and $V \in \mathfrak{X}_{\mV}$. 

\begin{theorem} \label{thbailocal}
Let  $\pi : (M, g_0) \rightarrow (N, g_N)$ be a Riemannian submersion. 
A 
variation $\{ g_t, t \in (-\epsilon, \epsilon) \} \subset \Riem(M, \mV , g_0)$ preserves the Riemannian submersion $\pi$
if and only if in every local adapted $g_0$-orthonormal frame $\{ E_a , \ee_i \}$ components $g_{ab}(t) = g_t(E_a, E_b)$, $g_{ai}(t)= g_t(E_a, \ee_i)$, $g_{ij}(t) = g_t(\ee_i ,\ee_j)$ of $g_t$ are the solution, on interval $t\in(-\epsilon, \epsilon)$, of 
the following system: 
\begin{eqnarray} \label{bxy1}
&& g_{ab} = \delta_{ab} \\ \label{bxy2} 
&& \dt g_{ai} = b_{ai} \\ \label{bxy3}
&& \dt g_{ij} = \sum\nolimits_{a=1}^n \left( b_{ia} g_{aj} + b_{ja} g_{ai} \right) ,
\end{eqnarray}
with initial conditions $g_{ai}(0)=0$ and $g_{ij}(0)=\delta_{ij}$, where $b_{ai}(t) = b_{ia}(t) = B_t(E_a, \ee_i)$. 
\end{theorem}
\begin{proof}
For an adapted $g_0$-orthonormal frame $\{E_a, \ee_i\}$, equation \eqref{bxy1} is equivalent to \eqref{BVVzero} and $\{E_a\}$ being $g_0$-orthonormal. 
Equation \eqref{bxy2} is equivalent to $\dt g_t(E_a, \ee_i) = B_t(E_a, \ee_i)$,
and \eqref{bxy3} is equivalent to \eqref{BXY}, written for $X=\ee_i$, $Y=\ee_j$. We note that, due to \eqref{BVVzero}, equation \eqref{BXY} written for $X=\ee_i$ and $Y=E_b$ yields \eqref{bxy2}, i.e., it does not impose any conditions on $\{ b_{ai} \}$.
\end{proof}

\begin{lemma} \label{lemlambdalocal}
Let  $\pi : (M, g_0) \rightarrow (N, g_N)$ be a Riemannian submersion. 
Let $W \subset M$ be an open set on which there exists an adapted $g_0$-orthonormal frame $\{ E_a , \ee_i \}$. Let $\{ \lambda_{ai} : W \rightarrow \mathbb{R} \}$ be smooth functions. Let $U$ be an open, relatively compact set, such that $\overline{U} \subset W$. Then there exists a variation $\{ g_t, t \in (-\epsilon, \epsilon) \} \subset \Riem(U, \mV , g_0)$, that preserves the Riemannian submersion $\pi$, such that $\{ g_{\mu \nu}(t) \}$ is the solution of the following system
\begin{eqnarray} \label{bsharp1}
&& g_{ab} = \delta_{ab} \\ \label{bsharp2}
&& \dt g_{ai} = \sum\nolimits_{j=n+1}^{n+p} \left( \lambda_{aj} g_{ij} - \sum\nolimits_{b=1}^n \lambda_{aj} g_{jb} g_{bi}  \right) \\
\label{bsharp3}
&& \dt g_{ij} = \sum\nolimits_{a=1}^n \sum\nolimits_{k=n+1}^{n+p} \big(
 g_{ik} g_{ja} \lambda_{ak} + g_{jk} g_{ia} \lambda_{ak} \nonumber \\
&& - \sum\nolimits_{c=1}^n \left( g_{jc} g_{ia} g_{kc} \lambda_{ak} + g_{ic} g_{ja} g_{kc} \lambda_{ak} \right) \big),
\end{eqnarray}
with $g_{\mu \nu}(0) = \delta_{\mu \nu}$, and $\{ P_{\mH}^t \ee_{n+1}, \ldots , P_{\mH}^t \ee_{n+p} \}$ are linearly independent for all $t \in (-\epsilon, \epsilon)$.
\end{lemma}
\begin{proof}
We consider \eqref{bsharp2}-\eqref{bsharp3} as a system of ODEs, smoothly depending on parameter $x \in W$.
It follows that for every $x \in W$ there exists a neighbourhood $\mN_x$ on which the solution $\{ g_{\mu \nu}(t)\}$ of \eqref{bsharp1}-\eqref{bsharp3} with $g_{ai}(0)=0$ and $g_{ij}(0)=\delta_{ij}$ exists on interval $t \in (-\epsilon_x, \epsilon_x)$ for some $\epsilon_x>0$. By decreasing $\epsilon_x$, if necessary, we can assume that $\{ g_{\mu \nu}(t) \}$ are components of positive-definite matrix for all $t \in (-\epsilon_x, \epsilon_x)$, because $g_{\mu \nu}(0)=\delta_{\mu \nu}$, and, since $\{ \ee_{n+1}, \ldots, \ee_{n+p}\}$ are linearly independent, $\{ \ee_i - \sum\nolimits_{a=1}^n g_{ia}(t) E_a , i =n+1, \ldots, n+p \}$ are linearly independent on $\mN_x$ for all $t \in (-\epsilon_x, \epsilon_x)$. Then $\{ \mN_x , x \in W \}$ is an open cover of the compact set ${\overline{U}}$, so it has a finite subcover 
and hence there exists $\epsilon>0$ such that solution $\{ g_{\mu \nu}(t)\}$ with all above properties exists for all $t \in (-\epsilon, \epsilon)$ on $U$. Hence, $\{ g_{\mu \nu}(t) \}$ define a Riemannian metric $g_t$ on $U$, determined by its components in the frame $\{E_a, \ee_i\}$ and bilinearity.

To prove that $\{ g_t, t \in (-\epsilon, \epsilon) \} \subset \Riem(U, \mV , g_0)$ 
preserves the Riemannian submersion $\pi$, observe that \eqref{BVVzero} holds by \eqref{bsharp1} 
and 
 \eqref{bsharp1}-\eqref{bsharp3} is a particular case of \eqref{bxy1}-\eqref{bxy3}, with 
\begin{equation} \label{bbsharp} 
b_{ai} =  \sum\nolimits_{j=n+1}^{n+p} \left( \lambda_{aj} g_{ij} - \sum\nolimits_{b=1}^n \lambda_{aj} g_{jb} g_{bi}  \right).
\end{equation}
Hence, \eqref{BXY} holds for all vector fields from the adapted frame $\{E_a, \ee_i\}$, and hence for all vector fields on $U$.
The last claim follows from $P_{\mH}^t \ee_i = \ee_i - \sum\nolimits_{a=1}^n g_{ia}(t) E_a$.
\end{proof}

\begin{remark} \label{remBinv}
We note that Riemannian metrics obtained as solutions of 
system 
\eqref{bsharp1}-\eqref{bsharp3} do not depend on the choice of adapted $g_0$-orthonormal frame $\{E_a, \ee_i\}$ if coefficients $\{ \lambda_{ai} \}$ are properly transformed when changing the frame. Indeed, any two such frames $\{E_a, \ee_i\}$ and $\{ E'_c , \ee'_k \}$ are related by a $g_0$-orthogonal transformation preserving each of distributions $\mV$ and $\mH(0)$, so $\ee'_k = \sum\nolimits_{i=n+1}^{n+p} M_{ik} \ee_i$, $E'_c = \sum\nolimits_{a=1}^{n} M_{a c} E_a$, with $M_{\mu \nu}$ being coefficients of a matrix $M \in O(n) \times O(p)$. 
Hence, e.g., 
$g'_{kc} = \sum\nolimits_{i=n+1}^{n+p} \sum\nolimits_{a=1}^{n} M_{ik} M_{ac} g_{ai}$. 
Let the transformation rule for $\{ \lambda_{ai} \}$ 
be the following:
\begin{equation} \label{lambdatransform}
\lambda'_{ck} = \sum\nolimits_{i=n+1}^{n+p} \sum\nolimits_{a=1}^{n} \lambda_{ai} M_{ik} M_{ac}. 
\end{equation} 
Comparing the system 
\eqref{bsharp1}-\eqref{bsharp3} 
written 
in frames $\{E_a, \ee_i\}$ and $\{ E'_c , \ee'_k \}$, we get $\dt g'_{kc} = \sum\nolimits_{i=n+1}^{n+p} \sum\nolimits_{a=1}^{n} \dt (M_{ik} M_{ac} g_{ai})$ and $\dt g'_{kl} = \sum\nolimits_{i=n+1}^{n+p} \sum\nolimits_{j=n+1}^{n+p} \dt ( M_{ik} M_{jl} g_{ij})$, so from the uniqueness of solution of ODE system, we have $g'_{kc}(t) = \sum\nolimits_{i=n+1}^{n+p} \sum\nolimits_{a=1}^{n} M_{ik} M_{ac} g_{ai}(t)$ and $g'_{kl}(t) = \sum\nolimits_{i=n+1}^{n+p} \sum\nolimits_{j=n+1}^{n+p} M_{ik} M_{jl} g_{ij}(t)$ for all $t$, for which the solution exists.
\end{remark}

\begin{theorem} \label{corbsharpglobal}
Let  $\pi : (M, g_0) \rightarrow (N, g_N)$ be a Riemannian submersion.
For every pair of sets: $\{ V_{n+1} , \ldots , V_{n+p} \}$ of vertical vector fields, 
bounded 
on $(M,g_0)$, and $\{ W_{n+1} , \ldots , W_{n+p} \}$ of linearly independent vector fields, 
bounded 
on $(N,g_N)$, there exists a unique 
variation $\{ g_t, t \in (-\epsilon, \epsilon) \} \subset \Riem(M, \mV , g_0)$ that preserves the Riemannian submersion $\pi$, such that for every $X \in \mathfrak{X}_{\mV}$
\begin{equation} \label{bsharpinv}
B_t^\sharp X = \sum\nolimits_{i=n+1}^{n+p} g_0(V_i , X) P_{\mH}^t  \pi^* W_i ,
\end{equation}
where $\pi^* W_i$ is the unique $g_0$-horizontal lift of $W_i$, for all $i \in \{n+1, \ldots , n+p\}$.
\end{theorem}
\begin{proof}
Let $\{ E_a , \ee_i \}$ be a local adapted 
$g_0$-orthonormal frame on an open set $U' \subset M$, and let $U$ be a relatively compact, open set, such that $\overline{U} \subset U'$. 
Then there exist functions $\{ \lambda_{ai}  : U' 
 \rightarrow \mathbb{R} \}$ such that
\begin{equation} \label{deflambdaai}
\sum\nolimits_{i=n+1}^{n+p} g_0(V_i , E_a) \pi^*(W_i) = \sum\nolimits_{i=n+1}^{n+p} \lambda_{ai} \ee_i .
\end{equation}
For these $\{\lambda_{ai}\}$, let $\{ g_t, t \in (-\epsilon', \epsilon') \} \subset \Riem(U, \mV , g_0)$ be the variation obtained from Lemma \ref{lemlambdalocal}.

From \eqref{bsharp2} it follows that coefficients $\{b_{ai}\}$ of $B_t = \dt g_t$ are given by \eqref{bbsharp}. From \eqref{bsharp1} we obtain that $B_t^\sharp E_a$ are $g_t$-horizontal, hence
\begin{equation} \label{Bsharptildelambda}
B^\sharp_t E_a = \sum\nolimits_{i=n+1}^{n+p} \left(  {\tilde \lambda}_{ai}(t) \ee_i - \sum\nolimits_{c=1}^n g_{ic}(t) {\tilde \lambda}_{ai}(t) E_c \right)
\end{equation}
for some $\{ \tilde \lambda_{ai} : U \times \mathbb{R} \rightarrow \mathbb{R} \}$. But obtaining $b_{ai} = B_t(E_a, \ee_i) = g_t(B^\sharp_t E_a , \ee_i)$ from \eqref{Bsharptildelambda} and comparing with \eqref{bbsharp}, we get for all $a \in \{1, \ldots ,n \}$ and all $i \in \{n+1 , \ldots, n+p\}$
\begin{equation} \label{lambdaunique}
\sum\nolimits_{j=n+1}^{n+p} ( {\tilde \lambda}_{aj}(t) - \lambda_{aj} ) ( g_{ij}(t) - \sum\nolimits_{b=1}^n g_{jb}(t) g_{bi}(t) ) = 0 .
\end{equation}
Since $g_{ij}(t) - \sum\nolimits_{b=1}^n g_{jb}(t) g_{bi}(t) = g_t( P_{\mH}^t  \ee_i , P_{\mH}^t  \ee_j)$ and, by Lemma \ref{lemlambdalocal}, 
$\{ P_{\mH}^t \ee_i \}$ are linearly independent, for all $t \in (-\epsilon', \epsilon')$ 
we have $\det g_t( P_{\mH}^t  \ee_i , P_{\mH}^t  \ee_j)  \neq 0$ and from \eqref{lambdaunique} it follows that ${\tilde \lambda}_{ai} = \lambda_{ai}$. Hence, 
\begin{equation} \label{bsharplambda} 
B^\sharp_t E_a = \sum\nolimits_{i=n+1}^{n+p}  \lambda_{ai} \left(  \ee_i - \sum\nolimits_{c=1}^n g_{ic}(t) E_c \right)
\end{equation} 
holds and, by \eqref{deflambdaai}, yields \eqref{bsharpinv}.

Now we prove that there exists $\epsilon>0$ such that at every $x \in M$ the solution of \eqref{bsharp1}-\eqref{bsharp3} exists for all $t \in (-\epsilon, \epsilon)$. Indeed, \eqref{bsharp1}-\eqref{bsharp3} are equations for coefficients of $g_t$ in a $g_0$-orthonormal frame, so the initial condition for these equations for all $x \in M$ is the same: $g_{ai}(0)=0$ and $g_{ij}(0)=\delta_{ij}$ for $a \in \{1, \ldots, n\}$ and $i,j \in \{n+1, \ldots, n+p\}$. Since $\{ V_i, W_i \}$ are bounded, from \eqref{deflambdaai} it follows that there exists a compact set $K \subset \mathbb{R}^{np}$, such that $\{ \lambda_{ai} \} \in K$ for all $x \in M$, 
and for all local $g_0$-orthonormal frames $\{E_a, \ee_i\}$. 
For every $\{ \lambda_{ai} \} \in K$ there exist $\epsilon' >0$ and an open neighbourhood $\mN \subset \mathbb{R}^{np}$ such that for all $\{ \lambda_{ai} \} \in \mN$ the solution of \eqref{bsharp1}-\eqref{bsharp3} with initial condition $g_{ai}(0)=0 , g_{ij}(0)=\delta_{ij}$ exists and, since the solutions depend smoothly on parameters $\{ \lambda_{ai} \}$ of the equations, 
$\det ( g_{ij}(t) - \sum\nolimits_{b=1}^n g_{jb}(t) g_{bi}(t) ) \neq 0$ for all $t \in (-\epsilon', \epsilon')$. Covering $K$ with these neighbourhoods and finding the smallest $\epsilon'$ for some finite subcover, we obtain an interval $(-\epsilon, \epsilon)$ on which $g_t$ is defined for all $x \in M$ as the metric whose components in a local frame $\{E_a , \ee_i \}$ are the solution of \eqref{bsharp1}-\eqref{bsharp3} with $\{ \lambda_{ai} \}$ defined by \eqref{deflambdaai}.
By Remark \ref{remBinv}, since $\{\lambda_{ai}\}$ defined by \eqref{deflambdaai}  satisfy the transformation rule \eqref{lambdatransform} and are obtained from globally defined vector fields, 
solutions obtained locally in adapted $g_0$-orthonormal frames give rise to a tensor field $g_t$ on $M$.
\end{proof}

\begin{remark} \label{remVit}
Let $U \subset N$ be an open set, on which linearly independent vector fields $W_{n+1}, \ldots , W_{n+p}$ are defined. Then on $\pi^{-1}(U)$ every one-parameter family of metrics
$\{ g_t, t \in (-\epsilon, \epsilon) \} \subset \Riem(M, \mV , g_0)$, that preserves the Riemannian submersion $\pi$, satisfies \eqref{bsharpinv} for some $t$-dependent vertical vector fields 
$\{ V_{n+1}(t) , \ldots , V_{n+p}(t) \}$ on $M$. 
This follows from $B_t^\sharp X$ being $g_t$-horizontal for every vertical $X$ and 
$\{ P_{\mH}^t \pi^* W_{n+1}, \ldots , P_{\mH}^t \pi^* W_{n+p} \}$ being a linear frame of distribution $\mH(t)$ on $\pi^{-1}(U)$.
\end{remark}

\begin{lemma} \cite{rz-2} \label{lemmaframe}
Let $\pi : (M,g_0) \rightarrow (N, g_N)$ be a Riemannian submersion and let 
$\{ g_t, t \in (-\epsilon, \epsilon) \} \subset \Riem(M, \mV , g_0)$ be a variation. 
If $\{ E_a, \ee_i \}$ is a local $g_0$-orthonormal frame on $M$ such that $E_a \in \mathfrak{X}_{\mV}$ for all $a \in \{ 1, \ldots ,n \}$, then $\{E_a(t) , \ee_i(t) \}$ such that
\begin{eqnarray*}
E_a(t) = E_a , \quad \ee_i(0) = \ee_i , \quad \dt \ee_i(t) = - \frac{1}{2} P_{\mH}^t  B_t^\sharp \ee(t) - P_{\mV}^t  B_t^\sharp \ee_i(t)
\end{eqnarray*}
is a local $g_t$-ortonormal frame on $M$.
\end{lemma} 

\begin{lemma} \label{lemmanablatwidetilde}
Let $\pi : (M,g_0) \rightarrow (N, g_N)$ be a Riemannian submersion and let 
$\{ g_t, t \in (-\epsilon, \epsilon) \} \subset \Riem(M, \mV , g_0)$ be a variation. Then for all $X,Y,Z \in \mathfrak{X}_{\mV}$ we have
\begin{equation} \label{dtgtXYZwidetildemD}
\dt g_t( \nabla^t_X Y , Z ) = 0
\end{equation}
and hence
\begin{equation} \label{P1P0eq}
P_{\mV}^t  ( \nabla^t_{X} {Y} + \nabla^t_{Y} {X}  )  = P_{\mV}^0 ( \nabla^0_{X} {Y} + \nabla^0_{Y} {X}  ) .
\end{equation}
\end{lemma}
\begin{proof}
Equation \eqref{dtgtXYZwidetildemD} follows from the Koszul formula \cite{KobayashiNomizu}
\begin{eqnarray*} 
&& 2 g_t (\nabla^t_X Y , Z) = X (g_t(Y,Z)) + Y (g_t(X,Z)) - Z (g_t(X,Y)) \nonumber \\
&& + g_t ( [X,Y] , Z ) + g_t([Z,X],Y) + g_t([Z,Y],X) 
= 2g_0 (\nabla^0_X Y, Z) .
\end{eqnarray*}
From \eqref{dtgtXYZwidetildemD} 
we obtain 
$\dt  P_{\mV}^t  ( \nabla^t_{X} {Y} + \nabla^t_{Y} {X}  ) = 0$, 
from which \eqref{P1P0eq} follows.
\end{proof}

\begin{lemma} \label{lemmadtnablaXYZRiem}
Let $\pi : (M, g_0) \rightarrow (N, g_N)$ be a Riemannian submersion and let variation $\{ g_t, t \in (-\epsilon, \epsilon) \} \subset \Riem(M, \mV, g_0)$ preserve the Riemannian submersion $\pi$.
Then for all $X,Y \in \mathfrak{X}_{\mV }$ 
and $Z \in \mathfrak{X}_{\mH(t)}$ we have
\begin{eqnarray} \label{dtnablaXYZ}
2 g_t (\dt \nabla^t_X Y , Z) 
&=& g_t([X, B_t^\sharp Y] , Z) 
 + g_t( [Y, B_t^\sharp X ], Z) \nonumber\\
&& - B_t( Z, P_{\mV}^0 ( \nabla^0_{X} {Y} + \nabla^0_{Y} {X}  ) ),
\end{eqnarray}
which is equivalent to
\begin{eqnarray} \label{dtnablaXYonly}
2 P_{\mH}^t \dt \nabla^t_{X} Y = P_{\mH}^t  \left( [X, B_t^\sharp Y] + [Y, B_t^\sharp X] - B_t^\sharp P_{\mV}^0 (\nabla^0_{X}Y +  \nabla^0_{Y}X ) \right).
\end{eqnarray}
\end{lemma}
\begin{proof}
We have, using \eqref{BVVzero}, 
\begin{eqnarray*}
&& 2 g_t (\dt \nabla^t_X Y , Z) = (\nabla^t_X B_t)(Y,Z) + (\nabla^t_Y B_t)(X,Z) - (\nabla^t_Z B_t)(X,Y) \\
&& =  X(B_t(Y,Z) ) - B_t( \nabla^t_X Y , Z  )- B_t( \nabla^t_X Z , Y ) 
+ Y(B_t(X,Z) ) - B_t( \nabla^t_Y X , Z  ) \\
&&\quad - B_t( \nabla^t_Y Z , X ) - Z(B_t(X,Y) ) + B_t( \nabla^t_Z X , Y ) + B_t( \nabla^t_Z Y , X ) \\ 
&&= g_t( \nabla^t_X B_t^\sharp Y, Z) + g_t (\nabla^t_X Z, B_t^\sharp Y)
- g_t( B_t^\sharp \nabla^t_X Z , Y ) + g_t( \nabla^t_Y B_t^\sharp X, Z)
\\&&\quad + g_t (\nabla^t_Y Z, B_t^\sharp X)
- g_t( B_t^\sharp \nabla^t_Y Z , X ) - g_t( \nabla^t_Z B_t^\sharp X, Y) - g_t (\nabla^t_Z Y, B_t^\sharp X) \\
&&\quad + g_t( B_t^\sharp \nabla^t_Z X , Y ) + g_t( B_t^\sharp \nabla^t_Z Y , X ) - g_t (B_t^\sharp Z , \nabla^t_X Y + \nabla^t_Y X ) \\
&& = g_t( \nabla^t_X B_t^\sharp Y, Z) 
 + g_t( \nabla^t_Y B_t^\sharp X, Z)
 - g_t( \nabla^t_Z B_t^\sharp X, Y) \\
&&\quad + g_t( B_t^\sharp \nabla^t_Z X , Y ) - g_t (B_t^\sharp Z , \nabla^t_X Y + \nabla^t_Y X ) .
\end{eqnarray*}
Using $Z \in \mathfrak{X}_{\mH(t)}$, $B_t^\sharp X ,  B_t^\sharp Y  \in \mathfrak{X}_{\mH(t)}$, 
${\tilde h}_t=0$ and $\tA=0$ for all $t$ (which follows from the fact that $\pi$ is a Riemannian submersion for all $g_t$), we obtain
\begin{eqnarray} \label{dtnablaXYZproof}
2 g_t (\dt \nabla^t_X Y , Z) 
&=& g_t( \nabla^t_X B_t^\sharp Y, Z) 
 + g_t( \nabla^t_Y B_t^\sharp X, Z)
 - g_t( \nabla^t_Z B_t^\sharp X, Y) \nonumber\\
&& + g_t( \nabla^t_Z X ,  B_t^\sharp Y ) - g_t (B_t^\sharp Z , \nabla^t_X Y + \nabla^t_Y X )\nonumber \\
&=& g_t( \nabla^t_X B_t^\sharp Y, Z) 
 + g_t( \nabla^t_Y B_t^\sharp X, Z)
 - g_t( {\tilde h}_t(Z , B_t^\sharp X) + {\tilde T}_t (Z , B_t^\sharp X) , Y) \nonumber\\
&& - g_t( \tA_X Z + \tT_X Z ,  B_t^\sharp Y ) - g_t (B_t^\sharp Z , \nabla^t_X Y + \nabla^t_Y X )\nonumber \\
&=& g_t( \nabla^t_X B_t^\sharp Y, Z) 
 + g_t( \nabla^t_Y B_t^\sharp X, Z)
 - g_t( \tT_Y Z , B_t^\sharp X ) \nonumber\\
&& - g_t( \tT_X Z ,  B_t^\sharp Y ) - g_t (B_t^\sharp Z , \nabla^t_X Y + \nabla^t_Y X ) \nonumber\\
&=& g_t( \nabla^t_{B_t^\sharp Y} X + [X, B_t^\sharp Y] , Z) 
 + g_t( \nabla^t_{B_t^\sharp X} Y + [Y, B_t^\sharp X ], Z) \nonumber \\&&
 - g_t( \tT_Y Z , B_t^\sharp X ) - g_t( \tT_X Z ,  B_t^\sharp Y ) - g_t (B_t^\sharp Z , \nabla^t_X Y + \nabla^t_Y X )\nonumber \\
&=& g_t( -\tT_X {B_t^\sharp Y} + [X, B_t^\sharp Y] , Z) 
 + g_t( -\tT_Y {B_t^\sharp X} + [Y, B_t^\sharp X ], Z) \nonumber \\&&
 - g_t( \tT_Y Z , B_t^\sharp X ) - g_t( \tT_X Z ,  B_t^\sharp Y ) - g_t (B_t^\sharp Z , \nabla^t_X Y + \nabla^t_Y X ) \nonumber\\
&=& g_t([X, B_t^\sharp Y] , Z) 
 + g_t( [Y, B_t^\sharp X ], Z)
 - g_t (B_t^\sharp Z , \nabla^t_X Y + \nabla^t_Y X ) .
\end{eqnarray}
For the last term 
in \eqref{dtnablaXYZproof}, we use \eqref{BXY}, assumption that $Z \in \mathfrak{X}_{\mH(t)}$ and \eqref{P1P0eq} to obtain
\begin{eqnarray} \label{dtnablaXYZproof2}
B_t(Z , \nabla^t_X Y + \nabla^t_Y X ) &=& \sum\nolimits_{a=1}^n B_t(Z , E_a) g_t(E_a , P_{\mV}^t  ( \nabla^t_X Y + \nabla^t_Y X  ) ) \nonumber \\ 
&& + \sum\nolimits_{a=1}^n g_t(Z , E_a) B_t(E_a , P_{\mV}^t  ( \nabla^t_X Y + \nabla^t_Y X  ) ) \nonumber\\
&=& \sum\nolimits_{a=1}^n B_t(Z , E_a) g_t(E_a , P_{\mV}^t  ( \nabla^t_X Y + \nabla^t_Y X  ) ) \nonumber\\
&=& B_t(Z ,  P_{\mV}^t  ( \nabla^t_X Y + \nabla^t_Y X  ) ) \nonumber\\
&=& B_t(Z ,  P_{\mV}^0 ( \nabla^0_X Y + \nabla^0_Y X  ) ) .
\end{eqnarray}
From \eqref{dtnablaXYZproof} and \eqref{dtnablaXYZproof2} we obtain \eqref{dtnablaXYZ}.
\end{proof}

\begin{lemma}
Let $\pi : (M, g_0) \rightarrow (N, g_N)$ be a Riemannian submersion and let variation $\{ g_t, t \in (-\epsilon, \epsilon) \} \subset \Riem(M, \mV, g_0)$ preserve the Riemannian submersion $\pi$. Then for every local $g_0$-orthonormal frame 
$\{E_1 , \ldots , E_n\} \subset \mathfrak{X}_{\mV}$ and for all $Z \in \mathfrak{X}_{\mH(t)}$ we have
\begin{equation} \label{gdtHZ}
g_t(\dt H_t, Z) = \sum\nolimits_{a=1}^n g_t([E_a, B_t^\sharp E_a] - B_t^\sharp ( P_{\mV}^0 (\nabla^0_{E_a} E_a) ), Z) ,
\end{equation}
which is equivalent to
\begin{equation} \label{P2dtH}
P_{\mH}^t  \dt H_t = P_{\mH}^t  \sum\nolimits_{a=1}^n \left( [E_a, B_t^\sharp E_a] - B_t^\sharp P_{\mV}^0 \nabla^0_{E_a} E_a \right) .
\end{equation}
\end{lemma}
\begin{proof}
We obtain \eqref{gdtHZ} by setting $X=Y=E_a$ in \eqref{dtnablaXYZ} 
and summing over $a \in \{1, \ldots , n\}$.
\end{proof} 

\subsection{Preserving extrinsic geometry of fibers} \label{secpresgeom}

We will write that a vector field is non-trivial if it is not everywhere zero, and that a variation $\{ g_t, t \in (-\epsilon, \epsilon) \} \subset \Riem(M, \mV, g_0)$ is non-trivial if $g_t \neq g_0$ for some $t \in (-\epsilon, \epsilon)$.

\begin{theorem} \label{thRSGFKillingglobal}
Let $\pi : (M, g_0) \rightarrow (N, g_N)$ be a Riemannian submersion with totally geodesic fibers. If there exists a non-trivial 
variation $\{ g_t, t \in (-\epsilon, \epsilon) \} \subset \Riem(M, \mV, g_0)$ 
such that $\pi : (M, g_t) \rightarrow (N, g_N)$ is a Riemannian submersion with totally geodesic fibers for all $t \in (-\epsilon, \epsilon)$, 
then 
there exists a non-trivial vertical vector field $\xi$, whose restriction to $(\mF_{x} , g_0 \vert_{\mF_{x}})$ is a Killing field for all $x \in N$.
On the other hand, if on $(M, g_0)$ 
exists a non-trivial bounded vertical 
vector field $\xi$, whose restriction to every fiber of $\pi$ is 
Killing, 
then there exists a non-trivial 
variation $g_t$ with the above properties.
\end{theorem}
\begin{proof}
Let $x \in N$. Let $U \subset N$ be a neighbourhood of $x$ on which there exist linearly independent vector fields $\{ W_{n+1}, \ldots , W_{n+p} \}$, then, by 
Remark \ref{remVit}, we can assume that \eqref{bsharpinv} holds on $\pi^{-1}(U)$. We will restrict our considerations to $\pi^{-1}(U)$.

For all $X \in \mathfrak{X}_{\mV}$ and all $i \in \{n+1,\ldots,n+p\}$, since $\pi^* W_i$ is projectable, we have $P_{\mH}^t [X,  \pi^* W_i ]=0$ and since $\mV$ is integrable, we have $P_{\mH}^t [X, P_{\mV}^t  \pi^* W_i ]=0$, so
\begin{equation} \label{lieXP2piW}
P_{\mH}^t  [X, P_{\mH}^t  \pi^* W_i] = P_{\mH}^t [X,  \pi^* W_i ] - P_{\mH}^t [X, P_{\mV}^t  \pi^* W_i ] = 0.
\end{equation}
Hence, from \eqref{dtnablaXYonly} and \eqref{bsharpinv} we obtain 
for all $X,Y \in \mathfrak{X}_{\mV}$
\begin{eqnarray} \label{dtPhdtnablaXYforVi}
&& 2 P_{\mH}^t \dt \nabla^t_{X} Y = P_{\mH}^t  \big( 
[X,  \sum\nolimits_{i=n+1}^{n+p} g_0(V_i(t) , Y) P_{\mH}^t  \pi^* W_i]  \nonumber \\
&&\quad + [Y,  \sum\nolimits_{i=n+1}^{n+p} g_0(V_i(t) , X) P_{\mH}^t  \pi^* W_i ] - B_t^\sharp P_{\mV}^0 (\nabla^0_{X}Y +  \nabla^0_{Y}X ) \big) \nonumber\\
&&= \sum\nolimits_{i=n+1}^{n+p} X(  g_0(V_i(t) , Y) ) P_{\mH}^t  \pi^* W_i  + \sum\nolimits_{i=n+1}^{n+p} Y(  g_0(V_i(t) , X) ) P_{\mH}^t  \pi^* W_i \nonumber\\
&&\quad - \sum\nolimits_{i=n+1}^{n+p} g_0(V_i(t) , \nabla^0_{X}Y +  \nabla^0_{Y}X) P_{\mH}^t  \pi^* W_i \nonumber\\
&&= \sum\nolimits_{i=n+1}^{n+p} \big( g_0(\nabla^0_X V_i(t) , Y) + g_0(\nabla^0_Y V_i(t) , X) \big) P_{\mH}^t  \pi^* W_i . 
\end{eqnarray}

The fibers of $\pi : (M, g_t) \rightarrow (N, g_N)$ remain totally geodesic if and only if $P_{\mH}^t \dt \nabla^t_{X} Y=0$ for all $X,Y \in \mathfrak{X}_{\mV}$. Indeed, using a frame from Lemma \ref{lemmaframe} we obtain
\begin{eqnarray*}
\dt P_{\mH}^t \nabla^t_{X} Y = P_{\mH}^t \dt \nabla^t_{X} Y - \sum\nolimits_{a=1}^n \sum\nolimits_{i=n+1}^{n+p} g_t( \nabla^t_{X} Y , \ee_i ) B_t(\ee_i , E_a) E_a.
\end{eqnarray*}
On the other hand, using \eqref{dtgtXYZwidetildemD} and \eqref{BXYfields}, we obtain
\begin{eqnarray*}
P_{\mV}^t \dt \nabla^t_{X} Y &=& \sum\nolimits_{a=1}^n g_t(\dt \nabla^t_{X} Y , E_a ) E_a = -\sum\nolimits_{a=1}^n B_t( \nabla^t_{X} Y  , E_a ) E_a \\
&=& -\sum\nolimits_{a=1}^n \sum\nolimits_{i=n+1}^{n+p} g_t( \nabla^t_{X} Y , \ee_i ) B_t( \ee_i , E_a ) E_a.
\end{eqnarray*}
It follows that $\dt P_{\mH}^t \nabla^t_{X} Y = P_{\mH}^t \dt \nabla^t_{X} Y + P_{\mV}^t \dt \nabla^t_{X} Y$. 
Hence, if $\dt h_t(X,Y) =0$ then $P_{\mH}^t \dt \nabla^t_{X} Y=0$.
On the other hand, by \eqref{BXYfields} and \eqref{P1P0eq}
\begin{eqnarray*}
\dt g_t(h_t(X,Y) , h_t(X,Y)) &=& B_t(h_t(X,Y) , h_t(X,Y)) + 2g_t( \dt P_{\mH}^t \nabla^t_{X} Y , P_{\mH}^t \nabla^t_{X} Y ) \\
&=& 2g_t( P_{\mH}^t \dt \nabla^t_{X} Y , P_{\mH}^t \nabla^t_{X} Y ) - 2g_t( P_{\mH}^t \dt P_{\mV}^t \nabla^t_{X} Y , P_{\mH}^t \nabla^t_{X} Y ) \\
&=& 2g_t( P_{\mH}^t \dt \nabla^t_{X} Y , P_{\mH}^t \nabla^t_{X} Y ) .
\end{eqnarray*}
Hence, if $P_{\mH}^t \dt \nabla^t_{X} Y =0$ and $h_0(X,Y)=0$, then $h_t(X,Y)=0$. 

Suppose that $g_t \neq g_0$ for some $t \in (-\epsilon, \epsilon)$ at some $y \in \mF_x$. 
If all $V_i(t)=0$ for all $t \in (-\epsilon, \epsilon)$, then $g_t = g_0$, so to have a non-trivial variation $g_t$ such that $P_{\mH}^t \dt \nabla^t_X Y=0$ for all $X,Y \in \mathfrak{X}_{\mV}$, at least one non-trivial vector field $\xi \in \{V_{n+1}(t) , \ldots , V_{n+p}(t) , t \in (-\epsilon, \epsilon) \}$ must satisfy
\[
g_0(\nabla^0_X \xi , Y) + g_0(\nabla^0_Y \xi , X)=0
\]
for all $X,Y \in \mathfrak{X}_{\mV}$, i.e., for all $x \in U$ the restriction of $\xi$ to $(\mF_{x} , g_0 \vert_{ \mF_{x} } )$ is a Killing vector field. Since $\xi \neq 0$ at some $y \in \pi^{-1}(U)$, 
let $\rho$ be a smooth function compactly supported in $U$ and such that $\rho(\pi(y))>0$, then $(\rho \circ \pi) \cdot \xi$ is a smooth vertical vector field on $M$, whose restriction to every fiber of $\pi$ is a Killing field.

On the other hand, suppose that on $(M,g_0)$ there exists a bounded non-trivial vertical vector field $\xi$, whose restriction to every fiber is Killing. Let $\{ W_{n+1} , \ldots, W_{n+p} \}$ be linearly independent vector fields on an open set $U \subset N$, such that $\xi \neq 0$ at some $x \in \pi^{-1}(U)$. We define a variation ${\tilde g}_t$ on $\pi^{-1}(U)$ by Theorem \ref{corbsharpglobal}, 
setting $V_{n+1} = \xi$ and $V_{n+2} = \ldots = V_{n+p} =0$, 
and then use a non-zero, compactly supported in $U$ function $\rho$, such that $0 \leq \rho \leq 1$ and $\rho(\pi(x))>0$, to define a variation $g_t$ on $M$ by $\dt g_t = (\rho \circ \pi) \cdot \dt {\tilde g}_t$.
\end{proof}

\begin{remark} \label{remextendKilling}
To define a variation that preserves a Riemannian submersion with totally geodesic fibers, it is sufficient to assume that there exists a bounded Killing vector field $\xi$ on a fiber 
$(\mF_x , g_0 \vert_{\mF_x})$ 
for one $x\in N$, 
because all fibers of a Riemannian submersion with totally geodesic fibers are isometric and we can extend such $\xi$ to a vertical field on $(M,g_0)$, whose restriction to every fiber is a Killing vector field. The isometries between totally geodesic fibers of $\pi$ on $(M, g_0)$ are defined by flows of $g_0$-horizontal lifts of vector fields from $N$ \cite{GromollWalschap}.
\end{remark}

\begin{theorem} \label{thRSMFdivfreeglobal}
Let $\pi : (M, g_0) \rightarrow (N, g_N)$ be a Riemannian submersion with minimal fibers. If there exists a non-trivial 
variation $\{ g_t, t \in (-\epsilon, \epsilon) \} \subset \Riem(M, \mV, g_0)$ 
such that $\pi : (M, g_t) \rightarrow (N, g_N)$ is a Riemannian submersion with minimal fibers for all $t \in (-\epsilon, \epsilon)$, 
then there exists a non-trivial vertical vector field $\xi$, whose restriction to $(\mF_{x} , g_0 \vert_{\mF_{x}})$ is divergence-free for all $x \in N$.
On the other hand, if on $(M, g_0)$ 
exists a non-trivial bounded vertical 
vector field 
$\xi$, whose restriction to every fiber of $\pi$ is divergence-free, 
then there exists a 
non-trivial variation $g_t$ with the above properties.
\end{theorem}
\begin{proof}
The proof is analogous to that of Theorem \ref{thRSGFKillingglobal} and based on the following computation:
we assume that \eqref{bsharpinv} holds and using 
\eqref{lieXP2piW}, we obtain from \eqref{P2dtH}
\begin{eqnarray} \label{P2dtHtlifts}
P_{\mH}^t  \dt H_t 
&=& P_{\mH}^t  \sum\nolimits_{a=1}^n \big( [E_a, \sum\nolimits_{i=n+1}^{n+p} g_0(V_i , E_a ) P_{\mH}^t  \pi^* W_i ] \nonumber \\
&& - \sum\nolimits_{i=n+1}^{n+p} g_0(V_i , P_{\mV}^0 \nabla^0_{E_a} E_a ) P_{\mH}^t  \pi^* W_i  \big) \nonumber\\
&=&  \sum\nolimits_{a=1}^n \sum\nolimits_{i=n+1}^{n+p}  \big( E_a ( g_0(V_i , E_a ) ) -  g_0(V_i , P_{\mV}^0 \nabla^0_{E_a} E_a ) \big) P_{\mH}^t  \pi^* W_i  \nonumber\\
&=& \sum\nolimits_{a=1}^n \sum\nolimits_{i=n+1}^{n+p}  g_0( \nabla^0_{E_a} V_i , E_a) P_{\mH}^t  \pi^* W_i  \nonumber\\
&=& \sum\nolimits_{i=n+1}^{n+p}  (\Div^0_{\mV} V_i ) P_{\mH}^t  \pi^* W_i .
\end{eqnarray}
\end{proof}

\begin{theorem} \label{thRSUFconfKillingglobal}
Let $\pi : (M, g_0) \rightarrow (N, g_N)$ be a Riemannian submersion with totally umbilical fibers. If there exists a 
non-trivial variation $\{ g_t, t \in (-\epsilon, \epsilon) \} \subset \Riem(M, \mV , g_0)$ 
such that $\pi : (M, g_t) \rightarrow (N, g_N)$ is a Riemannian submersion with totally umbilical fibers for all $t \in (-\epsilon, \epsilon)$, 
then there exists a non-trivial vertical vector field $\xi$, whose restriction to $(\mF_{x} , g_0 \vert_{\mF_{x}})$ is a conformal Killing field for all $x \in N$.
On the other hand, if 
on $(M,g_0)$ exists a non-trivial bounded vertical 
vector field
$\xi$, whose restriction to every fiber of $\pi$ is 
conformal Killing, 
then there exists a 
non-trivial variation $g_t$ with the above properties.
\end{theorem}
\begin{proof}
From 
\eqref{BVVzero},
we obtain for all $X,Y \in \mathfrak{X}_{\mV}$
\begin{eqnarray}
P_{\mH}^t  \dt ( g_t(X,Y) H_t ) &=& B_t(X,Y) P_{\mH}^t  H_t + g_t(X,Y) P_{\mH}^t  \dt H_t \nonumber\\
&=& g_t(X,Y) P_{\mH}^t  \dt H_t.
\end{eqnarray}
Hence, from \eqref{dtnablaXYonly} and \eqref{P2dtHtlifts} we obtain
\begin{eqnarray} 
2 P_{\mH}^t \dt ( \nabla^t_{X} Y - \frac{1}{n}g_t(X,Y)H_t ) &=& \sum\nolimits_{i=n+1}^{n+p} \big( g_0(\nabla^0_X V_i , Y) + g_0(\nabla^0_Y V_i , X) \nonumber \\
&& -\frac{2}{n} g_t(X,Y) \Div^0_{\mV} V_i \big) P_{\mH}^t  \pi^* W_i .
\end{eqnarray}
The proof then follows analogously to that of Theorem \ref{thRSGFKillingglobal}.
\end{proof}

\begin{remark}
Considering a compact manifold $(M,g)$ in Theorems \ref{thRSGFKillingglobal}-\ref{thRSUFconfKillingglobal} makes their formulation shorter, as vector fields $\xi$ are then always bounded. Also, if $\xi \in \mathfrak{X}_{\mF_x}$ is a conformal Killing vector field on a 
fiber $\mF_x$, it can be extended to a vector field, whose restriction to every fiber in some neighbourhood of $x$ is also conformal Killing - similarly as in Remark \ref{remextendKilling}, only now maps between totally umbilical fibers defined by flows of horizontal lifts of vector fields from $N$ will be conformal. An analogous reasoning applies to vertical divergence-free vector fields on minimal fibers, then the maps between minimal fibers preserve the volume form induced on every fiber, and hence also vertical divergence-free vector fields.
\end{remark}

\begin{theorem}
Let $\pi : (M, g_0) \rightarrow (N, g_N)$ be a Riemannian submersion with 
minimal fibers.
Then for every open set $U \subset M$ there exists a 
non-trivial variation $\{ g_t , t \in (-\epsilon, \epsilon) \} \subset \Riem(M, \mV ,g_0 )$ such that $g_t = g_0$ on $M \setminus U$ and $\pi : (M, g_t) \rightarrow (N, g_N)$ is a Riemannian submersion with minimal fibers for all $t \in (-\epsilon, \epsilon)$.
\end{theorem}
\begin{proof}
Let $x \in \pi(U)$ 
and let $\mP_x$ be the differential operator $\Div^0_{\mV}$ acting on the space $\mathfrak{X}_{\mF_x}$ 
of vector fields 
on the fiber $\mF_x$. 
Then $\mP_x$ is underdetermined elliptic, i.e., its symbol is surjective and non-injective 
\cite{BEM}.
Hence, $\mP_x$ 
has infinite-dimensional kernel $\mK_x$, which includes an infinite set of compactly supported vector fields 
on $\mF_x$ \cite{Delay}. Let 
$\{ V_{n+1} , \ldots , V_{n+p} \}$ be such vector fields. 
We can assume that $\pi(U)$ is a normal coordinate neighbourhood of $x$; for every $y \in \pi(U)$ let $\phi_{x,y} : \mF_x \rightarrow \mF_y$ be the map induced by the flow of the $g_0$-horizontal lift of the vector field tangent to the geodesic from $x$ to $y$. Since the fibers are minimal, the maps $\phi_{x,y}$ preserve the volume form induced on the fibers (see e.g., \cite[Chapter XV, Theorem 6.6]{Lang}). Hence, extending $\{ V_{n+1} , \ldots, V_{n+p} \}$ to $U$ by maps $\phi_{x,y *}$, we obtain a set of vertical vector fields, such that the restriction to every fiber in $U$ of each of them is divergence-free and compactly supported on the fiber.
Let $\rho$ be a smooth compactly supported function on $\pi(U)$. Multiplying each $V_i$ by $\rho \circ \pi$ 
we obtain a set of compactly supported, vertical vector fields in kernel of $\Div^0_{\mV}$, that by Theorem \ref{corbsharpglobal} defines a family $g_t$ with the required properties.
\end{proof}

\section{Variational problems} \label{secvarprobs} 

Let $\pi : (M, g_0) \rightarrow (N, g_N)$ be a Riemannian submersion and let $\vol_{g_t}$ be the volume form on $(M,g_t)$.
For $(0,2)$-tensor fields $Q,S$ on $M$, 
let
\[
\< Q , S \>_{\mV} = \sum\nolimits_{a=1}^n Q(E_a, E_b) S(E_a, E_b) ,
\]
where $\{E_a , a =1 ,\ldots ,n\} \subset \mathfrak{X}_{\mV}$ is a set of local $g_0$-orthonormal vector fields.

Let $J_{h}, J_{H} : \Riem(M , \mV , g_0) \rightarrow \mathbb{R}$ be functionals defined by formulas
\begin{eqnarray}
\label{Jh2def}
J_h(g_t) &=& \int_{M} \| h_t \|^2 \, \vol_{g_t} , \\ 
\label{JHdef}
J_H(g_t) &=& \int_{M} g_t(H_t,H_t) \, \vol_{g_t} , 
\end{eqnarray}
where $\| h_t \|^2 = \sum\nolimits_{a=1}^{n} g_t( h_t(E_a, E_b) , h_t(E_a, E_b) )$ for a local $g_t$-orthonormal frame $\{ E_a , a=1, \ldots,  n \} \subset \mathfrak{X}_{\mV}$. 
We say that a metric $g_0$ is critical for \eqref{Jh2def} 
(resp. 
\eqref{JHdef}) with respect to variations preserving the Riemannian submersion $\pi$, if for all 
variations $\{ g_t , t \in (-\epsilon,  \epsilon) \} \subset \Riem(M, \mV,  g_0)$, such that $\pi : (M, g_t) \rightarrow (N, g_N)$ is a Riemannian submersion for all $t \in (-\epsilon, \epsilon)$, we have $\dt J_h(g_t)  \vert_{t=0} = 0 $ (resp. $\dt J_H(g_t) \vert_{t=0} = 0$).

\begin{remark}
We note that all variations $\{ g_t , t \in (-\epsilon,  \epsilon) \} \subset \Riem(M, \mV,  g_0)$ preserving the Riemannian submersion $\pi$ also preserve the volume form $\vol_{g_0}$, since $\dt \vol_{g_t} = \frac{1}{2} \tr B_t^\sharp \vol_{g_t}$ \cite{Besse}. Indeed, let $\{E_a, \ee_i(t) \}$ be an adapted frame as in Lemma \ref{lemmaframe}, then by \eqref{BVVzero} and \eqref{BXY} we have
\begin{eqnarray*}
\tr B_t^\sharp &=& \sum\nolimits_{a=1}^{n} B_t(E_a,E_a) + \sum\nolimits_{i=n+1}^{n+p} B_t(\ee_i(t) , \ee_i(t)) \\
&=& 2 \sum\nolimits_{i=n+1}^{n+p} \sum\nolimits_{b=1}^{n} B_t(\ee_i(t) , E_b) g_t(E_b, \ee_i(t)) = 0.
\end{eqnarray*} 
\end{remark}

\subsection{Critical metrics} \label{seccritpoints}

We will use notation from Section \ref{secnotation}, especially its last paragraph. In particular, we have
$P_{\mH}^0 ( \delta_{0} h_0) (X) = P_{\mH}^0 \sum\nolimits_{a=1}^n (\nabla^0_{E_a} h_0)(E_a, X)$ for all $X \in \mathfrak{X}_{\mV}$ and any $g_0$-orthonormal local frame 
$\{E_1, \ldots , E_n\} \subset \mathfrak{X}_{\mV}$.

\begin{theorem} \label{thcrith}
Let $\pi : (M, g_0) \rightarrow (N, g_N)$ be a Riemannian submersion, 
where $M$ is compact.
Then $g_0$ is critical for 
\eqref{Jh2def}, with respect to variations preserving the Riemannian submersion $\pi$, if and only if $P_{\mH}^0 (\delta_{0} h_0) =0$.
\end{theorem} 
\begin{proof} 
Let $\{ g_t, t \in (-\epsilon, \epsilon) \} \subset \Riem(M, \mV, g_0)$ be a one-parameter family of Riemannian metrics on $M$ such that $\pi : (M, g_t) \rightarrow (N, g_N)$ is a Riemannian submersion for all $t \in (-\epsilon, \epsilon)$. 

Let $\{ U_\alpha , \alpha \in \{1 ,\ldots, \Lambda\} \}$ be a finite open cover of $N$,
such that on every $U_\alpha$ there exists a 
frame of linearly independent vector fields $\{W_{n+1}^\alpha, \ldots, W_{n+p}^\alpha \}$ and on $\pi^{-1}(U_\alpha)$ we have $\vol_{g_0} = \pi^* (\vol_{g_N}(x)) \otimes \vol_{\mF_x, g_0}$ for all $x \in U_\alpha$ \cite{Lang}, where $\vol_{\mF_x, g_0}$ is the volume form of $\mF_x$ induced from $(M, g_0)$. 
Then, by Remark \ref{remVit}, there exist vertical vector fields $\{V_{n+1}^\alpha (t), \ldots, V_{n+p}^\alpha (t) \}$ on $\pi^{-1}(U_\alpha)$ such that
\[
B_t^\sharp X = \sum\nolimits_{i=n+1}^{n+p} g_0(V_i^\alpha (t) , X) P_{\mH}^t  \pi^* W_i^\alpha
\]
for all vertical vector fields $X$ on $\pi^{-1}(U_\alpha)$.
Let $z \in \pi^{-1}(U_\alpha)$, then 
for every set of local $g_0$-orthonormal vector fields $\{E_a , a=1, \ldots, n\} \subset \mathfrak{X}_{\mV}$ in a neighbourhood of $z$
we have, by \eqref{dtPhdtnablaXYforVi}, at the point $z$:
\begin{eqnarray} \label{dtht}
&& \dt \|h_t\|^2 = B_t(h_t(E_a, E_b) , h_t(E_a, E_b)) + 2g_t(\dt h_t(E_a, E_b) , h_t(E_a, E_b)) \nonumber \\
&&= 2g_t( (\dt \nabla^t)_{E_a} E_b , h_t(E_a, E_b) ) = 2g_t( P_{\mH}^t (\dt \nabla^t)_{E_a} E_b , h_t(E_a, E_b) ) \nonumber \\
&&= \sum\nolimits_{i=n+1}^{n+p} \big( g_0(\nabla^0_{E_a} V^\alpha_i (t) , E_b) + g_0(\nabla^0_{E_b} V^\alpha_i (t) , E_a) \big) g_t( P_{\mH}^t  \pi^* W_i^\alpha , h_t(E_a, E_b)) \nonumber \\
&&= \sum\nolimits_{i=n+1}^{n+p}  \< \mathcal{L}_{V^\alpha_i (t)} g_0 , h^{i,\alpha}_t \>_{\mV} ,
\end{eqnarray}
where $h_t^{i,\alpha}(X,Y) = g_t( h_t(X, Y) , P_{\mH}^t  \pi^* W_i^\alpha )$ for all vertical vector fields $X,Y$ on $\pi^{-1}(U_\alpha)$. 
Let $\{ f_\alpha \}$ be a partition of unity 
on $N$, subordinate to $\{ U_\alpha \}$, then \cite{Lang} 
\begin{eqnarray} \label{dtinth2partition}
&& \dt \int_{M} \| h_t \|^2 \, \vol_{g_t} \vert_{t=0} = \int_{M} \dt \| h_t \|^2 \, \vol_{g_0} \vert_{t=0} \nonumber \\
&&=  \sum\nolimits_{\alpha=1}^\Lambda  \int_{\pi^{-1}(U_\alpha)} f_\alpha  \dt \| h_t \|^2 \, \vol_{g_0} \vert_{t=0} \nonumber\\ 
&&= \sum\nolimits_{\alpha=1}^\Lambda \sum\nolimits_{i=n+1}^{n+p} \int_{\pi^{-1}(U_\alpha)} f_\alpha \< \mathcal{L}_{V_i^\alpha (0)} g_0 , h^{i, \alpha}_0\>_{\mV} \, \vol_{g_0} \nonumber\\
&& = \sum\nolimits_{\alpha=1}^\Lambda \sum\nolimits_{i=n+1}^{n+p} \int_{U_\alpha} \big( \int_{\mF_x} \< \mathcal{L}_{V_i^\alpha (0)} g_0 , h^{i,\alpha}_0\>_{\mV}  \, \vol_{\mF_x, g_0} \big) f_\alpha 
\vol_{g_N}(x) .
\end{eqnarray}

We will prove that $\dt J_h (g_t) \vert_{t=0} =0$ for all variations $\{ g_t, t \in (-\epsilon, \epsilon) \} \subset \Riem(M, \mV, g_0)$ preserving the Riemannian submersion $\pi$ if and only if for all $x \in N$, all $V \in \mathfrak{X}_{\mF_x}$ and all $W \in T_x N$ we have
\begin{equation} \label{Jhcrit1fiber}
\int_{\mF_x} \< \mathcal{L}_{V} g_0 , h^W_0\>_{\mV} \, \vol_{\mF_x , g_0} =0 ,
\end{equation}
where $h^W_t(X,Y) = g_t( h_t(X , Y) ,  P_{\mH}^t \pi^* W )$ for all $X,Y \in \mathfrak{X}_{\mF_x}$, where $\pi^* W$ is the $g_0$-horizontal lift of $W$.

The sufficiency of \eqref{Jhcrit1fiber} follows from \eqref{dtinth2partition}.
For the necessity, 
if \eqref{Jhcrit1fiber} does not hold for some $W \in T_x$ and $V \in \mathfrak{X}_{\mF_x}$,
then (since a Riemannian submersion from a compact manifold is a fiber bundle \cite{Ehresmann}) there exists an open set $U \subset N$ with $x \in U$, such that we can extend $V$ to a vertical field $V_{n+1}$ on $\pi^{-1}(U)$ and 
there exists a linearly independent frame $\{W_{n+1} , \ldots, W_{n+p}\}$ with $W_{n+1}=W$ on $U$ and 
\begin{equation*} 
\int_{\pi^{-1}(U)} \< \mathcal{L}_{V} g_0 , h^W_0\>_{\mV} \, \vol_{g_0} \neq 0.
\end{equation*}
We then set $V_{n+2} = \ldots = V_{n+p}=0$ and define a variation ${\tilde g}_t$ on $\pi^{-1}(U)$ as in Theorem \ref{corbsharpglobal}. Using a non-zero, compactly supported in $U$ function $\rho$, such that $0 \leq \rho \leq 1$ we can define a variation $g_t$ on $M$ by $\dt g_t = \rho \cdot \dt {\tilde g}_t$, for which $\dt J_h (g_t) \vert_{t=0} \neq 0$.

Let $(V^\flat)(X) = g_0(V, X)$ for all $V,X \in \mathfrak{X}_{\mF_x}$, then we can write \eqref{Jhcrit1fiber} in the following form (see the last paragraph in Section \ref{secnotation}):
\[
0 = \int_{\mF_x} \< 2 \delta_0^* (V^\flat) , h^W_0\>_{\mV} \, \vol_{\mF_x , g_0} = 2 \int_{\mF_x} \< V^\flat , \delta_0 h^W_0\>_{\mV} \, \vol_{\mF_x , g_0}
\]
and since the above must hold for all $V \in \mathfrak{X}_{\mF_x}$, we obtain $\delta_0 h^W_0 =0$ for all $W \in T_x M$ as the necessary and sufficient condition for $g_0$ to be a critical point of \eqref{Jh2def} with respect to variations preserving the Riemannian submersion $\pi$. 

Let $W$ be an extension of $W \in T_x N$ to $U$, 
let $\{ E_a , a=1,\ldots,n \}$ be a local vertical $g_0$-orthonormal frame. Using $P_{\mH}^0 [E_a, \pi^* W] =0$, we obtain for all vertical vector fields $X$ on $\pi^{-1}(U)$:
\begin{eqnarray*}
(\delta_0 h_0^W) (X) &=& \sum\nolimits_{a=1}^n ( g_0 ( (\nabla^0_{E_a} h_0)(E_a, X) ,  \pi^* W) + g_0(h_0(E_a, X) ,\nabla_{E_a} \pi^* W ) ) \\
&=& g_0( (\delta_0 h_0)(X) ,  \pi^* W) +  \sum\nolimits_{a=1}^n g_0(h_0(E_a, X) ,\nabla_{\pi^* W} E_a ) \\
&=& g_0( (\delta_0 h_0)(X) ,  \pi^* W) + \sum\nolimits_{a=1}^n g_0( h_0(E_a, X) ,\nabla_{\pi^* W} E_a ) \\
&=& g_0( (\delta_0 h_0)(X) ,  \pi^* W) - \sum\nolimits_{a=1}^n g_0( {\tilde h}_0( \pi^* W , h_0(E_a, X) ) , E_a ) \\
&=& g_0( (\delta_0 h_0)(X) ,  \pi^* W).
\end{eqnarray*}
Hence, $(\delta_0 h_0^W)=0$ for all $W \in \mathfrak{X}_N$ if and only if $P_{\mH}^0 (\delta_0 h_0) =0$.
\end{proof} 

\begin{remark} 
We note that for all $X \in \mathfrak{X}_{\mV}$ 
and for a local $g_0$-orthonormal adapted frame $\{E_a, \ee_i\}$, 
we can write the Codazzi formula \cite{Tondeur} in the following form: 
\[
P_{\mH}^0 (\delta_0 h_0)(X) = P_{\mH}^0 \sum\nolimits_{a=1}^n R(E_a , X) E_a + P_{\mH}^0 \nabla_X H_0 ,
\]
where $R$ is the curvature tensor of $g_0$. 
\end{remark} 

\begin{corollary} \label{corcmc}
Let $\pi : (M, g_0) \rightarrow (N, g_N)$ be a Riemannian submersion, 
where $M$ is compact. If the fibers of $\pi$ are totally umbilical on $(M,g_0)$, 
then $g_0$ is critical for 
\eqref{Jh2def}, with respect to variations preserving the Riemannian submersion $\pi$, if and only if all fibers of $\pi$ are of parallel mean curvature on $(M, g_0)$, i.e., $P_{\mH}^0 \nabla^0_X H_0=0$ for all $X \in \mathfrak{X}_{\mV}$.
\end{corollary} 
\begin{proof}
For totally umbilical fibers 
$P_{\mH}^0 (\delta_0 h_0)(X) = \frac{1}{n} P_{\mH}^0 \nabla^0_X H_0$ for all $X \in \mathfrak{X}_{\mV}$.
\end{proof} 

\begin{theorem} \label{thcritnormH}
Let $\pi : (M, g_0) \rightarrow (N, g_N)$ be a Riemannian submersion, 
where $M$ is compact.
Then $g_0$ is critical for 
\eqref{JHdef}, with respect to variations preserving the Riemannian submersion $\pi$, if and only if $H_0$ is 
projectable.
\end{theorem}
\begin{proof}
Similarly as in the proof of Theorem \ref{thcrith}, we obtain that $g_0$ is critical if and only if for all $x \in N$, 
$V \in \mathfrak{X}_{\mF_x}$, 
$W \in T_x N$ we have
\begin{equation} \label{H0critfibers}
\int_{\mF_x}  g_0(H_0 , 
\pi^* W ) \cdot (\Div^0_{\mV} V) \; \vol_{\mF_x , g_0 } =0.
\end{equation}
Using $(V^\flat)(X) = g_0(V, X)$ for all $V,X \in \mathfrak{X}_{\mF_x}$, the exterior differential ${\rm d}$ and its formal adjoint ${\rm d}^*$ on $\mF_x$, we can write \eqref{H0critfibers} as
\[
0 = \int_{\mF_x}  g_0(H_0 , 
\pi^* W ) \cdot  {\rm d}^* (V^\flat) \; \vol_{\mF_x , g_0 } =  -\int_{\mF_x} ( {\rm d} g_0(H_0 , 
\pi^* W ) )(V) \; \vol_{\mF_x , g_0 }.
\]
It follows that $g_0$ is critical for \eqref{JHdef} with respect to variations preserving the Riemannian submersion $\pi$ if and only if $P_{\mV}^0 \nabla^0 g_0(H_0 , \pi^* W )=0$ 
for all $W \in \mathfrak{X}_N$. Writing $H_0 = \sum\nolimits_{i=n+1}^{n+p} g_0( H_0 , \pi^* W_i) \pi^* W_i$ for a local $g_N$-orthonormal basis $\{W_{n+1} , \ldots,  W_{n+p}\}$, we find that this condition is equivalent to $H_0$ being 
projectable, as for all $V \in \mathfrak{X}_{\mV}$
\begin{eqnarray*}
P_{\mH}^0 [V , H_0] &=& \sum\nolimits_{i=n+1}^{n+p} V( g_0( H_0 , \pi^* W_i) ) \pi^* W_i + \sum\nolimits_{i=n+1}^{n+p} g_0( H_0 , \pi^* W_i) P_{\mH}^0 [V , \pi^* W_i ] \\
&=& \sum\nolimits_{i=n+1}^{n+p} ( P_{\mV}^0 \nabla^0 g_0( H_0 , \pi^* W_i) , V ) \pi^* W_i.
\end{eqnarray*}
\end{proof}

\begin{remark}
We note that for a fixed $x \in N$, we can define functionals 
by \eqref{Jh2def} and \eqref{JHdef}, with $M$ in their definitions replaced by the fiber $\mF_x$, i.e., 
\begin{eqnarray}
\label{Jh2defx}
J_{h,x}(g_t) &=& \int_{\mF_x} \| h_t \|^2 \,  \vol_{\mF_x , g_t } , \\ 
\label{JHdefx}
J_{H,x}(g_t) &=& \int_{\mF_x} g_t(H_t,H_t) \,  \vol_{\mF_x , g_t } .
\end{eqnarray}
The critical metrics for \eqref{Jh2defx} and \eqref{JHdefx} satisfy the same conditions as those given in Theorems \ref{thcrith} and \ref{thcritnormH}, only on the fiber $\mF_x$ rather than everywhere on $M$.
\end{remark}

\subsection{Second variation} \label{secsecondvar}
\begin{lemma} 
Let $\{ g_t, t \in (-\epsilon, \epsilon) \} \subset \Riem(M, \mV, g_0)$ be a 
variation such that $\pi : (M, g_t) \rightarrow (N, g_N)$ is a Riemannian submersion for all $t \in (-\epsilon, \epsilon)$ and \eqref{bsharpinv} holds, where $\{ W_{n+1}, \ldots , W_{n+p} \}$ are $g_N$-orthonormal. Then for all $t \in (-\epsilon, \epsilon)$ the vector fields $\{ P_{\mH}^t \pi^* W_{n+1} , \ldots , P_{\mH}^t \pi^* W_{n+p} \}$ are $g_t$-orthonormal and $g_t$-horizontal, 
and
\begin{equation} \label{dtPHWi}
\dt P_{\mH}^t \pi^* W_i = -V_i 
\end{equation}
for all $i \in \{ n+1 , \ldots, n+p\}$.
\end{lemma}
\begin{proof}
The first statement follows from 
\begin{equation} \label{gNWiWj}
g_t(P_{\mH}^t \pi^* W_i , P_{\mH}^t \pi^* W_j ) = g_N(W_i, W_j) = \delta_{ij}.
\end{equation}
To prove the last one, we have for all $X \in \mathfrak{X}_M$:
\begin{eqnarray} \label{dtPVX}
\dt P_{\mV}^t X &=& \dt \sum\nolimits_{a=1}^{n}  g_t(X , E_a) E_a = \sum\nolimits_{a=1}^{n} \big( B_t(X , E_a) E_a + g_t(\dt X , E_a) E_a \big) \nonumber \\
&=& \sum\nolimits_{a=1}^{n} g_t( B_t^\sharp E_a , X) E_a + P_{\mV}^t (\dt X) \nonumber\\
&=& \sum\nolimits_{a=1}^{n} \sum\nolimits_{i=n+1}^{n+p} g_0(V_i , E_a) g_t( P_{\mH}^t \pi^* W_i , X) E_a + P_{\mV}^t (\dt X)  \nonumber\\
&=& \sum\nolimits_{i=n+1}^{n+p} g_t( P_{\mH}^t \pi^* W_i , X) V_i + P_{\mV}^t (\dt X) .
\end{eqnarray}
In particular,
\begin{eqnarray*}
\dt P_{\mH}^t \pi^* W_i &=& - \dt P_{\mV}^t \pi^* W_i
= -\sum\nolimits_{j=n+1}^{n+p} g_t( P_{\mH}^t \pi^* W_i , \pi^* W_j) V_j \\
&=& - \sum\nolimits_{i=n+1}^{n+p} g_N( W_i , W_j)  V_j = -V_i.
\end{eqnarray*}
\end{proof}

\begin{theorem} \label{thd2tHgeq0}
Let $\pi : (M, g_0) \rightarrow (N, g_N)$ be a Riemannian submersion, where $M$ is compact. Let $g_0$ be a critical point of \eqref{Jh2def}, $($respectively, \eqref{JHdef}$)$, 
with respect to variations 
preserving the Riemannian submersion $\pi$. Then for all such variations we have $\ddt J_{h}(g_t) \geq 0$ $($respectively, $\ddt J_{H}(g_t) \geq 0$ $)$.
\end{theorem} 
\begin{proof}
Let $\{ g_t, t \in (-\epsilon, \epsilon) \} \subset \Riem(M, \mV, g_0)$ be a one-parameter family of Riemannian metrics on $M$ 
that preserve Riemannian submersion $\pi$ and let $x \in N$. By Remark \ref{remVit}, we can assume that \eqref{bsharpinv} holds on $\pi^{-1}(U)$, for an open neighbourhood $U$ of $x$ in $N$, 
with 
$\{ W_{n+1}, \ldots , W_{n+p} \}$ that are $g_N$-orthonormal at $x$ and $t$-dependent $\{ V_{n+1}, \ldots , V_{n+p}  \}$.

Let $X,Y \in \mathfrak{X}_{\mF_x}$, from \eqref{dtPhdtnablaXYforVi} we obtain
\begin{equation} \label{dthtXY}
P_{\mH}^t \dt h_t(X,Y) = \frac{1}{2} \sum\nolimits_{i=n+1}^{n+p} (\mathcal{L}_{V_i(t)} g_0 )(X,Y) P_{\mH}^t \pi^* W_i .
\end{equation}
Let $h_t^{i}(X,Y) = g_t( h_t(X, Y) , P_{\mH}^t  \pi^* W_i )$, similarly as \eqref{dtht} we obtain
\[
\dt \| h_t \|^2 = \sum\nolimits_{i=n+1}^{n+p}  \< \mathcal{L}_{V_i(t)} g_0 , h^{i}_t \>_{\mV}
\]
and from \eqref{BXYfields}, \eqref{gNWiWj} and \eqref{dthtXY} it follows that
\begin{eqnarray} \label{dthiXY}
\dt h^i_t(X,Y) &=& \dt g_t(h_t(X,Y) , P_{\mH}^t \pi^* W_i) =  g_t( \dt h_t(X,Y) , P_{\mH}^t \pi^* W_i) \nonumber \\
&=& \frac{1}{2} \sum\nolimits_{j=n+1}^{n+p} (\mathcal{L}_{V_j(t)} g_0 )(X,Y)  g_t( P_{\mH}^t \pi^* W_j , P_{\mH}^t \pi^* W_i) \nonumber \\
&=& \frac{1}{2} (\mathcal{L}_{V_i(t)} g_0 )(X,Y) .
\end{eqnarray}
Hence,
\begin{eqnarray} \label{ddtnormh}
\ddt \| h_t \|^2 &=& \sum\nolimits_{i=n+1}^{n+p}  \< \mathcal{L}_{\dt V_i} g_0 , h^{i}_t \>_{\mV} + \sum\nolimits_{i=n+1}^{n+p}  \< \mathcal{L}_{V_i(t)} g_0 , \dt h^{i}_t \>_{\mV} \nonumber \\
&=& \sum\nolimits_{i=n+1}^{n+p}  \< \mathcal{L}_{\dt V_i} g_0 , h^{i}_t \>_{\mV} + \frac{1}{2} \sum\nolimits_{i=n+1}^{n+p}  \< \mathcal{L}_{V_i(t)} g_0 , \mathcal{L}_{V_i(t)} g_0  \>_{\mV}  .
\end{eqnarray}
Similarly, we obtain on $\mF_x$:
\begin{eqnarray}
P_{\mH}^t \dt H_t &=& \sum\nolimits_{i=n+1}^{n+p} (\Div^0_{\mV} V_i(t))  P_{\mH}^t \pi^* W_i  , \nonumber \\
\dt g_t(H_t, H_t) &=& 2 \sum\nolimits_{i=n+1}^{n+p} (\Div^0_{\mV} V_i(t)) g_t( P_{\mH}^t \pi^* W_i, H_t) , \nonumber \\
\label{ddtgHH}
\ddt  g_t(H_t, H_t)
&=& 2 \sum\nolimits_{i=n+1}^{n+p} (\Div^0_{\mV} \dt V_i) g_t( P_{\mH}^t \pi^* W_i, H_t) \nonumber\\
&& + 2 \sum\nolimits_{i=n+1}^{n+p} (\Div^0_{\mV} V_i(t))^2 .
\end{eqnarray}

If $g_0$ is a critical point of \eqref{Jh2def}, with respect to variations preserving the Riemannian submersion $\pi$, then \eqref{Jhcrit1fiber} holds and hence from \eqref{ddtnormh} we obtain
\begin{eqnarray} \label{ddtnormhtgeq0}
&& \ddt \int_{\mF_x} \| h_t \|^2 \, \vol_{g_t} \vert_{t=0} = \int_{\mF_x} \ddt \| h_t \|^2 \, \vol_{g_t}  \vert_{t=0} \nonumber \\
&&= \int_{\mF_x} \sum\nolimits_{i=n+1}^{n+p}  \< \mathcal{L}_{\dt V_i \vert_{t=0}} g_0 , h^{i}_0 \>_{\mV} \, \vol_{g_0} + \frac{1}{2} \int_{\mF_x} \sum\nolimits_{i=n+1}^{n+p}  \< \mathcal{L}_{V_i(0)} g_0 , \mathcal{L}_{V_i(0)} g_0  \>_{\mV}  \, \vol_{g_0} \nonumber \\
&&= \frac{1}{2} \int_{\mF_x} \sum\nolimits_{i=n+1}^{n+p}  \< \mathcal{L}_{V_i(0)} g_0 , \mathcal{L}_{V_i(0)} g_0  \>_{\mV}  \, \vol_{g_0} \geq 0.
\end{eqnarray}
Since the above holds for all $x \in N$, we have $\ddt \int_{M} \| h_t \|^2 \, \vol_{g_t} \vert_{t=0} \geq 0$, by the same method as used in \eqref{dtinth2partition}.

If $g_0$ is a critical point of \eqref{JHdef}, with respect to variations preserving the Riemannian submersion $\pi$, then \eqref{H0critfibers} holds and hence from $P_{\mH}^t \pi^* W_i \vert_{t=0} =  \pi^* W_i$ and \eqref{ddtgHH}, we obtain 
\begin{eqnarray} \label{ddtgHHgeq0}
&& \ddt \int_{\mF_x} g_t(H_t, H_t) \, \vol_{g_t} \vert_{t=0} = \int_{\mF_x} \ddt g_t(H_t, H_t) \, \vol_{g_t}  \vert_{t=0} \nonumber \\
&&= 2 \int_{\mF_x} \sum\nolimits_{i=n+1}^{n+p} g_0( 
\pi^* W_i, H_0) (\Div^0_{\mV} (\dt V_i \vert_{t=0}) ) \, \vol_{g_0} \nonumber \\
&&\quad + 2\int_{\mF_x} \sum\nolimits_{i=n+1}^{n+p} (\Div^0_{\mV} V_i(0))^2  \, \vol_{g_0} \nonumber \\
&&= 2 \int_{\mF_x} \sum\nolimits_{i=n+1}^{n+p} (\Div^0_{\mV} V_i(0))^2  \, \vol_{g_0} \geq 0.
\end{eqnarray}
Again, since \eqref{ddtgHHgeq0} holds for all $x \in N$, we have $ \ddt \int_{M} g_t(H_t, H_t) \, \vol_{g_t} \vert_{t=0} \geq 0$.
\end{proof}

\begin{theorem} \label{thd2tHgeq0case0}
Let $\pi : (M, g_0) \rightarrow (N, g_N)$ be a Riemannian submersion, where $M$ is compact. Let $x \in N$ and let $g_0$ be a critical point of \eqref{Jh2defx}, 
with respect to variations 
preserving the Riemannian submersion $\pi$.
Let  $\{ g_t, t \in (-\epsilon, \epsilon) \} \subset \Riem(M, \mV, g_0)$ be a variation that preserves the Riemannian submersion $\pi$ and let $U$ be an open subset on $N$ such that $x\in U$ and \eqref{bsharpinv} holds on $\pi^{-1}(U)$ for some $t$-dependent vector fields $V_i(t)$. 

Let $q \geq 2$. We have $\frac{\partial^r }{\partial t^r} J_{h,x} (g_t) \vert_{t=0} =0$ for all $1 \leq r \leq q$ if and only if $\{
V_i^{(s)}(0) , i=n+1, \ldots, n+p \}$ are Killing vector fields on 
$(\mF_x, g_0 \vert_{\mF_x})$ 
for all $0 \leq s \leq \lfloor \frac{q}{2} \rfloor -1$, where we use notation $V_i^{(s)}(0) = \frac{\partial^s }{\partial t^s} V_i \vert_{t=0}$ for all $s \geq 1$ and $V_i^{(0)}(0) = V_i(0)$. 
\end{theorem} 
\begin{proof}
We have $\frac{\partial }{\partial t} J_{h,x} (g_t) \vert_{t=0} =0$ by the assumption of $g_0$ being critical. 
For all $r \geq 2$, from differentiating \eqref{ddtnormh} and using \eqref{dthiXY}, it follows that $\frac{\partial^r }{\partial t^r} J_{h,x}(g_t) \vert_{t=0}$ is a linear combination, with positive coefficients, of term
\begin{equation} \label{termLvgh}
\int_{\mF_x} \sum\nolimits_{i=n+1}^{n+p} \< \mathcal{L}_{  
V^{(r-1)}_i(0) } g_0 , h^{i}_0 \>_{\mV} \, \vol_{g_0} ,
\end{equation} 
which vanishes by \eqref{Jhcrit1fiber}, and terms of the form 
\begin{equation} \label{termLvg}
\int_{\mF_x} \sum\nolimits_{i=n+1}^{n+p} \< \mathcal{L}_{  
V^{(\alpha)}_i(0) 
} g_0 , \mathcal{L}_{  
V^{(\beta)}_i(0)
} g_0 \>_{\mV} \, \vol_{g_0},
\end{equation}
for all $\alpha,\beta \geq 0$ such that $\alpha+\beta = r-2$. In particular, in every term \eqref{termLvg} we have either $\alpha \leq \frac{r}{2} - 1$ or $\beta \leq \frac{r}{2} - 1$, so if $\{ 
V^{(s)}_i(0), i=n+1,\ldots, n+p \}$ are Killing vector fields on 
$(\mF_x, g_0 \vert_{\mF_x})$ 
for all $0 \leq s \leq \lfloor \frac{r}{2} \rfloor -1$ then $\frac{\partial^r }{\partial t^r} J_{h,x}(g_t) \vert_{t=0}=0$.

On the other hand, suppose that $\frac{\partial^r }{\partial t^r} J_{h,x} (g_t) \vert_{t=0} =0$ for all $1\leq r \leq q$ for some $q \geq 2$. From \eqref{ddtnormhtgeq0} and $q \geq 2$, we obtain that $\{ V_i(0) , i=n+1, \ldots, n+p \}$ are Killing vector fields on $(\mF_x, g_0 \vert_{\mF_x})$. 
Suppose that 
$\{ 
V^{(\alpha)}_i(0), i= n+1, \ldots, n+p \} $ are Killing vector fields on $(\mF_x, g_0 \vert_{\mF_x})$, 
for all $0 \leq \alpha \leq s$, where $s \geq 0$. Then $\frac{\partial^{2s+2} }{\partial t^{2s+2}} J_{h,x} (g_t) \vert_{t=0}$ is a linear combination, with positive coefficients, of term \eqref{termLvgh} with $r=2s+2$, which vanishes by \eqref{Jhcrit1fiber}; terms \eqref{termLvg}, where either $\alpha \leq s$ or $\beta \leq s$, which vanish by the assumption; and term 
\[ 
\int_{\mF_x} \sum\nolimits_{i=n+1}^{n+p} \<\mathcal{L}_{  
V^{(s+1)}_i(0) } g_0 ,\mathcal{L}_{  
V^{(s+1)}_i(0) } g_0  \>_{\mV} \, \vol_{g_0} .
\]
If $\frac{\partial^{2s+2} }{\partial t^{2s+2}} J_{h,x} (g_t) \vert_{t=0}=0$, then the above term also vanishes, and hence 
$\{ 
V^{(s+1)}_i(0) , i=n+1, \ldots, n+p \}$ are Killing vector fields on $(\mF_x, g_0 \vert_{\mF_x})$. By induction on $s$, we obtain that 
$\{ 
V^{(\alpha)}_i(0), i=n+1, \ldots, n+p \}$ are Killing vector fields on $(\mF_x, g_0 \vert_{\mF_x})$ for all $\alpha \leq s$, as long as $2s+2 \leq q$, i.e., $s \leq \lfloor \frac{q}{2} \rfloor -1$.
\end{proof}

\begin{remark}
A theorem analogous to Theorem \ref{thd2tHgeq0case0} can be proved for functional \eqref{JHdefx}, with $J_{h,x}$ replaced by $J_{H,x}$ and 
words ``Killing vector fields on $(\mF_x, g_0 \vert_{\mF_x})$'' replaced by ``divergence-free vector fields on $(\mF_x, g_0 \vert_{\mF_x})$''.
\end{remark}

\section{Sectional curvature} \label{secseccurv}

\subsection{Horizontal curvature}  \label{seccurvXY}

In this section we compute variations, with respect to metrics in $\Riem(M , \mV , g_0)$ preserving a Riemannian submersion $\pi$, of sectional curvatures in directions of horizontal lifts of given vector fields on $N$. Formulas for these variations are not tensorial and hence the 
critical points are only those metrics, for which the integrability tensor of 
the horizontal distribution vanishes on the lifts. 

\begin{theorem} \label{thd2tTgeq0}
Let $\{ g_t, t \in (-\epsilon, \epsilon) \} \subset \Riem(M, \mV, g_0)$ be a 
variation such that $\pi : (M, g_t) \rightarrow (N, g_N)$ is a Riemannian submersion for all $t \in (-\epsilon, \epsilon)$ and \eqref{bsharpinv} holds, where $\{ W_{n+1}, \ldots , W_{n+p} \}$ are $g_N$-orthonormal vector fields on $N$. Let $\sec_M(X,Y)$ be the sectional curvature of the plane field 
spanned by linearly independent $X,Y \in \mathfrak{X}_M$.
Then
\begin{eqnarray} \label{dtnormTWiWjsec}
&& \dt \sec_M( P_{\mH}^t \pi^* W_i , P_{\mH}^t \pi^* W_j ) \nonumber \\
&&= -\frac{3}{2} \sum\nolimits_{k=n+1}^{n+p} g_N( [ W_i , W_j ] , W_k ) g_t( V_k ,   P_{\mV}^t [ P_{\mH}^t \pi^* W_i , P_{\mH}^t \pi^* W_j  ] ) \nonumber\\
&& \quad + \frac{3}{2}  g_t( [V_i , P_{\mH}^t \pi^* W_j  ] + [ P_{\mH}^t \pi^* W_i , V_j ] , P_{\mV}^t [ P_{\mH}^t \pi^* W_i , P_{\mH}^t \pi^* W_j  ] ) .
\end{eqnarray}
and
\begin{eqnarray} \label{d2tnormTWiWjsec}
&& \ddt \sec_M( P_{\mH}^t \pi^* W_i , P_{\mH}^t \pi^* W_j ) \nonumber \\
&&= -\frac{3}{2} \sum\nolimits_{k=n+1}^{n+p} g_N( [ W_i , W_j ] , W_k ) g_t( \dt V_k ,  P_{\mV}^t [ P_{\mH}^t \pi^* W_i , P_{\mH}^t \pi^* W_j  ] ) \nonumber \\
&&\quad - \frac{3}{2} \sum\nolimits_{k,l=n+1}^{n+p} g_N( [ W_i , W_j ] , W_k )  g_N( [ W_i , W_j ] , W_l ) g_t( V_k , V_l ) \nonumber\\
&&\quad +3 \sum\nolimits_{k=n+1}^{n+p} g_N( [ W_i , W_j ] , W_k ) g_t( V_k , [ V_i , P_{\mH}^t \pi^* W_j  ] + [ P_{\mH}^t \pi^* W_i , V_j ] ) \nonumber\\
&&\quad + \frac{3}{2} g_t( [\dt V_i , P_{\mH}^t \pi^* W_j  ] + [ P_{\mH}^t \pi^* W_i , \dt V_j ] , P_{\mV}^t [ P_{\mH}^t \pi^* W_i , P_{\mH}^t \pi^* W_j  ] )\nonumber\\
&&\quad - 3 g_t( [ V_i , V_j  ]  , P_{\mV}^t [ P_{\mH}^t \pi^* W_i , P_{\mH}^t \pi^* W_j  ] ) \nonumber\\
&&\quad - \frac{3}{2} g_t( [ V_i , P_{\mH}^t \pi^* W_j  ] + [ P_{\mH}^t \pi^* W_i ,  V_j ] , [ V_i , P_{\mH}^t \pi^* W_j ] + [ P_{\mH}^t \pi^* W_i ,  V_j ] ) .
\end{eqnarray}
\end{theorem}
\begin{proof} 
For all $i \in \{n+1, \ldots, n+p\}$, from \eqref{lieXP2piW} it follows that $P_{\mH}^t \pi^* W_i$ is 
projectable, so $P_{\mH}^t \pi^* W_i$ is the $g_t$-horizontal lift of $W_i$. By the Jacobi identity, the Lie bracket of projectable fields is also projectable. Hence, 
\begin{eqnarray} \label{liebracketWiWj}
&&g_t( [ P_{\mH}^t \pi^* W_i , P_{\mH}^t \pi^* W_j  ] , P_{\mH}^t \pi^* W_k ) 
= g_N( \pi_* [ P_{\mH}^t \pi^* W_i , P_{\mH}^t \pi^* W_j  ] , W_k ) \nonumber \\
&&=g_N( [ \pi_* P_{\mH}^t \pi^* W_i , \pi_* P_{\mH}^t \pi^* W_j  ] , W_k ) = g_N( [ W_i , W_j ] , W_k ) .
\end{eqnarray}

For all $U \in \mathfrak{X}_{\mV}$ we have, by 
\eqref{BVVzero}, \eqref{dtPVX} and \eqref{liebracketWiWj}
\begin{eqnarray} \label{dtTWiWjU}
&&\dt g_t (  P_{\mV}^t [ P_{\mH}^t \pi^* W_i , P_{\mH}^t \pi^* W_j  ] , U ) \nonumber \\
&&= g_t ( \dt P_{\mV}^t [ P_{\mH}^t \pi^* W_i , P_{\mH}^t \pi^* W_j  ] , U ) \nonumber \\
&&= \sum\nolimits_{k=n+1}^{n+p} g_t( [ P_{\mH}^t \pi^* W_i , P_{\mH}^t \pi^* W_j  ] , P_{\mH}^t \pi^* W_k ) g_t( V_k , U ) \nonumber \\
&& \quad - g_t( [V_i , P_{\mH}^t \pi^* W_j  ] + [ P_{\mH}^t \pi^* W_i , V_j ] , U ) \nonumber \\ 
&&= \sum\nolimits_{k=n+1}^{n+p} g_N( [ W_i , W_j ] , W_k ) g_0( V_k , U ) \nonumber \\
&& \quad - g_t( [V_i , P_{\mH}^t \pi^* W_j  ] + [ P_{\mH}^t \pi^* W_i , V_j ] , U ) .
\end{eqnarray}

Hence,
\begin{eqnarray} \label{dtnormTWiWj}
&& \dt g_t( P_{\mV}^t [ P_{\mH}^t \pi^* W_i , P_{\mH}^t \pi^* W_j  ] ,   P_{\mV}^t [ P_{\mH}^t \pi^* W_i , P_{\mH}^t \pi^* W_j  ] ) \nonumber \\
&&= 2 g_t( \dt P_{\mV}^t [ P_{\mH}^t \pi^* W_i , P_{\mH}^t \pi^* W_j  ] ,   P_{\mV}^t [ P_{\mH}^t \pi^* W_i , P_{\mH}^t \pi^* W_j  ] ) \nonumber\\
&&= 2 \sum\nolimits_{k=n+1}^{n+p} g_N( [ W_i , W_j ] , W_k ) g_t( V_k ,   P_{\mV}^t [ P_{\mH}^t \pi^* W_i , P_{\mH}^t \pi^* W_j  ] ) \nonumber\\
&& \quad - 2g_t( [V_i , P_{\mH}^t \pi^* W_j  ] + [ P_{\mH}^t \pi^* W_i , V_j ] , P_{\mV}^t [ P_{\mH}^t \pi^* W_i , P_{\mH}^t \pi^* W_j  ] ) 
\end{eqnarray}
and
\begin{eqnarray} \label{d2tnormTWiWj}
&& \ddt g_t( P_{\mV}^t [ P_{\mH}^t \pi^* W_i , P_{\mH}^t \pi^* W_j  ] ,   P_{\mV}^t [ P_{\mH}^t \pi^* W_i , P_{\mH}^t \pi^* W_j  ] ) \nonumber \\
&&= 2 \sum\nolimits_{k=n+1}^{n+p} g_N( [ W_i , W_j ] , W_k ) g_t( \dt V_k ,  P_{\mV}^t [ P_{\mH}^t \pi^* W_i , P_{\mH}^t \pi^* W_j  ] ) \nonumber \\
&&\quad + 2 \sum\nolimits_{k,l=n+1}^{n+p} g_N( [ W_i , W_j ] , W_k )  g_N( [ W_i , W_j ] , W_l ) g_t( V_k , V_l ) \nonumber\\
&&\quad - 4 \sum\nolimits_{k=n+1}^{n+p} g_N( [ W_i , W_j ] , W_k ) g_t( V_k , [ V_i , P_{\mH}^t \pi^* W_j  ] + [ P_{\mH}^t \pi^* W_i , V_j ] ) \nonumber\\
&&\quad - 2g_t( [\dt V_i , P_{\mH}^t \pi^* W_j  ] + [ P_{\mH}^t \pi^* W_i , \dt V_j ] , P_{\mV}^t [ P_{\mH}^t \pi^* W_i , P_{\mH}^t \pi^* W_j  ] )\nonumber\\
&&\quad + 2g_t( [ V_i , V_j  ] + [ V_i , V_j ] , P_{\mV}^t [ P_{\mH}^t \pi^* W_i , P_{\mH}^t \pi^* W_j  ] ) \nonumber\\
&&\quad + 2g_t( [ V_i , P_{\mH}^t \pi^* W_j  ] + [ P_{\mH}^t \pi^* W_i ,  V_j ] , [ V_i , P_{\mH}^t \pi^* W_j ] + [ P_{\mH}^t \pi^* W_i ,  V_j ] ) .
\end{eqnarray}
For a Riemannian submersion we have \cite{O'Neill}
\begin{eqnarray} \label{seccurvRS}
&& \sec_M( P_{\mH}^t \pi^* W_i, P_{\mH}^t \pi^* W_j  ) = \sec_N(W_i, W_j) \nonumber \\
&& - \frac{3}{4} g_t( P_{\mV}^t [P_{\mH}^t \pi^* W_i , P_{\mH}^t \pi^* W_j] , P_{\mV}^t [P_{\mH}^t \pi^* W_i , P_{\mH}^t \pi^* W_j] ).
\end{eqnarray}
Using \eqref{seccurvRS}, we obtain \eqref{dtnormTWiWjsec} and \eqref{d2tnormTWiWjsec} from \eqref{dtnormTWiWj} and \eqref{d2tnormTWiWj}. 
\end{proof}

Setting $[W_i, W_j]=0$, $\dt V_i =0$ and $[V_i, V_j]=0$ for all $i,j \in \{n+1, \ldots , n+p\}$ in \eqref{d2tnormTWiWjsec}, we obtain the following.

\begin{corollary}
Let $\{ g_t, t \in (-\epsilon, \epsilon) \} \subset \Riem(M, \mV, g_0)$ be a 
variation such that $\pi : (M, g_t) \rightarrow (N, g_N)$ is a Riemannian submersion for all $t \in (-\epsilon, \epsilon)$ and \eqref{bsharpinv} holds, where 
$\{W_{n+1} , \ldots, W_{n+p}\}$ are $g_N$-orthonormal.
If $[W_i, W_j]=0$, $V_i$ do not depend on $t$ and $[V_i, V_j] =0$ for all $i,j \in \{n+1,\ldots,n+p\}$, then $\ddt \sec_M (P_{\mH}^t \pi^* W_i , P_{\mH}^t \pi^* W_j) \leq 0$. 
\end{corollary} 

\begin{corollary}
Let $\pi : (M, g_0) \rightarrow (N, g_N)$ be a Riemannian submersion. 
Let $U \subset N$ be an open set and let 
$\{W_{n+1} , \ldots, W_{n+p}\}$ 
be $g_N$-orthonormal vector fields on $U$. Then $g_0$ is a critical point of the functional
\begin{equation} \label{JsecU}
J_{\sec(W_i,W_j),U} : g_t \mapsto \int_{\pi^{-1}(U)} \sec(P_{\mH}^t \pi^* W_i , P_{\mH}^t \pi^* W_j) \vol_{g_t},
\end{equation}
with respect to variations preserving the Riemannian submersion $\pi$,
if and only if we have ${\tilde T}_0 ( \pi^* W_i , \pi^* W_j ) =0$ on $\pi^{-1}(U)$,
where $\pi^* W_i, \pi^* W_j$ are $g_0$-horizontal lifts of $W_i, W_j$ (respectively).
\end{corollary}
\begin{proof}
From \eqref{dtnormTWiWj} and \eqref{seccurvRS} it follows that if  ${\tilde T}_0 ( \pi^* W_i , \pi^* W_j ) =0$, then $g_0$ is critical for \eqref{JsecU}.
On the other hand, by 
\eqref{dtnormTWiWjsec} we have 
\[
\dt \sec_M( P_{\mH}^t \pi^* W_i , P_{\mH}^t \pi^* W_j ) = I_i (V_i) + I_j (V_j) + \sum_{k=n+1, k \neq i , k \neq j}^{n+p} J_k (V_k)
\]
where
\begin{eqnarray*}
I_i (V_i) &=& 2 g_N( [ W_i , W_j ] , W_i ) g_t( V_i ,   P_{\mV}^t [ P_{\mH}^t \pi^* W_i , P_{\mH}^t \pi^* W_j  ] ) \\
&& - 2g_t( [V_i , P_{\mH}^t \pi^* W_j  ] , P_{\mV}^t [ P_{\mH}^t \pi^* W_i , P_{\mH}^t \pi^* W_j  ] ) ,
\end{eqnarray*}
\begin{eqnarray*}
I_j (V_j) &=& 2 g_N( [ W_i , W_j ] , W_j ) g_t( V_j ,   P_{\mV}^t [ P_{\mH}^t \pi^* W_i , P_{\mH}^t \pi^* W_j  ] ) \\
&& - 2g_t( [ P_{\mH}^t \pi^* W_i , V_j ] , P_{\mV}^t [ P_{\mH}^t \pi^* W_i , P_{\mH}^t \pi^* W_j  ] )
\end{eqnarray*}
and
\[
J_k (V_k) =  2 
g_N( [ W_i , W_j ] , W_k ) g_t( V_k ,   P_{\mV}^t [ P_{\mH}^t \pi^* W_i , P_{\mH}^t \pi^* W_j  ] ) .
\]
Since for $f \in C^\infty(M)$ we have 
\[
I_i(fV_i) = fI_i(V_i) + 2g_t( V_i , P_{\mV}^t [ P_{\mH}^t \pi^* W_i , P_{\mH}^t \pi^* W_j  ] ) \cdot (P_{\mH}^t \pi^* W_j)(f),
\]
it follows that if ${\tilde T}_0 ( \pi^* W_i , \pi^* W_j ) \neq 0$, then by considering various $f$, we can make values of $\dt \sec_M( P_{\mH}^t \pi^* W_i , P_{\mH}^t \pi^* W_j )$ arbitrary.
\end{proof}

\begin{remark}
We note that for a local $g_N$-orthonormal frame $\{ W_{n+1} , \ldots W_{n+p} \}$ on $N$ we can define on $(M, g_t)$ the scalar $g_t$-horizontal curvature $K_{\mH(t)} = \sum\nolimits_{i,j=n+1}^{n+p} \sec(P_{\mH}^t \pi^* W_i , P_{\mH}^t \pi^* W_j)$. By \eqref{seccurvRS}, we have $K_{\mH (t)} = K_N - \frac{3}{4} \| {\tilde T}_t \|^2$, where $K_N$ is the scalar curvature of $(N, g_N)$ and $\| {\tilde T}_t \|^2 = \sum\nolimits_{i=n+1}^{n+p} g_t( P_{\mV}^t [ P_{\mH}^t \pi^* W_i , P_{\mH}^t \pi^* W_j ] , P_{\mV}^t [ P_{\mH}^t \pi^* W_i , P_{\mH}^t \pi^* W_j ] )$. Thus, varying $K_{\mH(t)}$ is equivalent to varying $\| {\tilde T}_t \|^2$. 
Examples of critical points of functional $g_t \mapsto \int_M \| {\tilde T}_t \|^2 \, \vol_{g_t}$  on a domain of a Riemannian submersion $\pi : (M,g) \rightarrow (N, g_N)$ are metrics, for which $(M,g)$ is a $3$-Sasakian (which was proved in \cite{rz-2}), or a $f$-$K$-contact manifold (which can be proved by similar computations as those in \cite{TZCMS}), in each case with vertical characteristic vector fields.
\end{remark}


\subsection{Vertizontal curvature of totally geodesic fibers} \label{seccurvXU}

The curvatures $\sec_M (X,U)$ for horizontal $X$ and vertical $U$ are called \emph{vertizontal}, and a Riemannian submersion with totally geodesic fibers and all vertizontal curvatures positive is called \emph{fat} \cite{ZillerFatness}. In this section we compute variations of vertizontal curvature preserving a Riemannian submersion with totally geodesic fibers and show that on certain manifolds (e.g., spheres $S^{4m+3}$ with the Hopf fibration by $S^3$) we can make this variation non-constant on a fiber.

\begin{theorem} \label{thd2tsecXUgeq0}
Let $\{ g_t, t \in (-\epsilon, \epsilon) \} \subset \Riem(M, \mV, g_0)$ be a 
variation such that $\pi : (M, g_t) \rightarrow (N, g_N)$ is a Riemannian submersion for all $t \in (-\epsilon, \epsilon)$ and \eqref{bsharpinv} holds, where 
$\{W_{n+1}, \ldots , W_{n+p}\}$ are $g_N$-orthonormal and 
$\{V_{n+1}, \ldots , V_{n+p}\}$ are bounded vertical vector fields whose restrictions to every fiber $\mF_x$ are Killing fields on $(\mF_x ,g_0 \vert_{\mF_x})$. Let the fibers of $\pi : (M, g_0) \rightarrow (N, g_N)$ be totally geodesic. 
Then for all $U \in \mathfrak{X}_{\mV}$ we have
\begin{eqnarray} \label{dtsecXURSGF}
\dt \sec_M (P_{\mH}^t \pi^* W_i,U) &=& \frac{1}{2} \sum\nolimits_{j=n+1}^{n+p} g_t( [P_{\mH}^t \pi^* W_i, P_{\mH}^t \pi^* W_j], U) \cdot \nonumber\\
&& \cdot \big( 
\sum\nolimits_{k=n+1}^{n+p} g_N( [ W_i , W_j ] , W_k ) g_0( V_k , U ) \nonumber\\
&&  - g_t( [V_i , P_{\mH}^t \pi^* W_j  ] + [ P_{\mH}^t \pi^* W_i , V_j ] , U ) 
\big) .
\end{eqnarray}
and
\begin{eqnarray} \label{ddtsecXURSGF}
&& \ddt \sec_M( P_{\mH}^t \pi^* W_i , U) \nonumber\\ 
&&=\frac{1}{2} 
\sum\nolimits_{j=n+1}^{n+p} 
\big( 
\sum\nolimits_{l=n+1}^{n+p} g_N( [ W_i , W_j ] , W_l ) g_0( V_l , U ) 
- g_t( [V_i , P_{\mH}^t \pi^* W_j  ] \nonumber\\
&& + [ P_{\mH}^t \pi^* W_i , V_j ] , U ) 
\big)^2 
- \sum\nolimits_{j=n+1}^{n+p} g_t( [P_{\mH}^t \pi^* W_i, P_{\mH}^t \pi^* W_j], U) \cdot \big(  g_t( [\dt V_i , P_{\mH}^t \pi^* W_j  ] , U) \nonumber \\
&& + g_t( [\dt V_j , P_{\mH}^t \pi^* W_i  ] , U) \big) .
\end{eqnarray}
\end{theorem} 
\begin{proof} 
By Theorem \ref{thRSGFKillingglobal}, for all $t \in (-\epsilon, \epsilon)$, $\pi : (M, g_t) \rightarrow (N, g_N)$ is a Riemannian submersion with totally geodesic fibers. For such submersions we have \cite{O'Neill} 
\begin{equation} \label{RSGFsecXU} 
4 \sec_M (P_{\mH}^t \pi^* W_i,U) = \sum\nolimits_{j=n+1}^{n+p} g_t( [P_{\mH}^t \pi^* W_i, P_{\mH}^t \pi^* W_j], U)^2
\end{equation} 
and from \eqref{dtTWiWjU} 
we obtain \eqref{dtsecXURSGF}.
Using \eqref{lieXP2piW}, \eqref{BVVzero} and \eqref{dtPHWi} we obtain
\begin{equation} \label{dtgtViPHWjU}
\dt g_t( [V_i , P_{\mH}^t \pi^* W_j  ] , U) = B_t ( [V_i , P_{\mH}^t \pi^* W_j  ] , U ) -  g_t( [V_i , V_j  ] , U) = -  g_t( [V_i , V_j  ] , U) .
\end{equation}
Applying the above together with  $[V_i, V_j] = -[V_j, V_i]$ in \eqref{dtsecXURSGF} yields \eqref{ddtsecXURSGF}.
\end{proof} 

\begin{corollary} \label{ddtsecXUgeq0}
With the assumptions of Theorem \ref{thd2tsecXUgeq0}, if $\dt V_i = 0$ for all $i \in \{ n+1 , \ldots, n+p \}$, then $\ddt \sec_M( P_{\mH}^t \pi^* W_i , U) \geq 0$ for all $i \in \{ n+1 , \ldots, n+p \}$.
\end{corollary}

\begin{corollary} \label{corpressecUXex}
Let $\dim N \geq 2$, let $\pi : (M, g_0) \rightarrow (N, g_N)$ be a Riemannian submersion with totally geodesic fibers and non-integrable horizontal distribution, such that every fiber $\mF_x$ is spanned by orthonormal vertical vector fields $\{ E_1, \ldots, E_n \}$, whose restrictions to $\mF_x$ are Killing vector fields on $(\mF_x, g_0 \vert_{\mF_x})$. 

For $W \in T_x N$, let $\pi^* W$ be the $g_0$-horizontal lift of $W$. 
If there exists a fiber $\mF_y$ such that for some $X,Y \in T_y N$:
\begin{itemize}
\item[$1^{\circ}$] $g_0( [ \pi^* X, \pi^* Y ] ,E_1 )$ is not constant on $\mF_y$, or 
\item[$2^{\circ}$] $g_0( [ \pi^* X, \pi^* Y ] ,E_1 ) \neq 0$ on $\mF_y$ and there exists a Killing vector field $U_y$ on $(\mF_y ,g_0 \vert_{\mF_y})$ such that $g_0(U_y, E_1)$ is not constant on $\mF_y$,
\end{itemize}
then there exists a metric $g_{s}$ such that $\pi : (M, g_s) \rightarrow (N, g_N)$ is a Riemannian submersion with totally geodesic fibers and sectional curvature $\sec_M( P_{\mH}^s \pi^* X , E_1)$ not constant on the fiber $\mF_y$.
\end{corollary}
\begin{proof}
Let ${\cal O}$ be an open neighbourhood of $y$ on which there exist $g_N$-orthonormal vector fields $\{W_{n+1}, \ldots , W_{n+p}\}$ such that at the point $y$ we have: $W_{n+1} = X$, $W_{n+2} = Y$ and $[W_i, W_j]=0$ for all $i,j \in \{n+1, \ldots , n+p\}$.
Let $f \in C^\infty(N)$ be such that 
$Y(f) \neq 0$ and $W_{j}(f) = 0$ at $y$ for all $j \in \{n+1, n+3, \ldots, n+p\}$.

If case $1^{\circ}$ holds, let $\xi = E_1$. If case $2^{\circ}$ holds, let $\xi = U$, where $U$ is a vector field whose restriction to a fiber $\mF_x$ is a Killing vector field on $(\mF_x, g_0 \vert_{\mF_x})$ for all $x \in {\cal O}$, and such that $U = U_y$ on $\mF_y$ ($U_y$ can be extended to such vector field $U$ by Remark \ref{remextendKilling}). 

Let $V_{n+1} = (f \circ \pi) \cdot \xi$ and let $V_{n+2} = \ldots = V_{n+p} =0$ on ${\cal O}$.
Let $\{ {\tilde g}_t, t \in (-\epsilon, \epsilon) \} \subset \Riem(M, \mV, g_0)$ be a 
variation such that $\pi : (\pi^{-1}({\cal O}), {\tilde g}_t) \rightarrow ({\cal O}, g_N)$ is a Riemannian submersion for all $t \in (-\epsilon, \epsilon)$ and \eqref{bsharpinv} holds on $\pi^{-1}({\cal O})$, with 
$\{W_{n+1}, \ldots , W_{n+p}\}$ and $\{V_{n+1}, \ldots , V_{n+p} \}$ as above.
Let $\rho \in C^\infty(N)$ be a non-zero, compactly supported in ${\cal O}$ function, such that 
$\rho=1$ on some neighbourhood of $y$. From now on we will consider the variation $g_t$ with $\dt g_t = (\rho \circ \pi) \cdot \dt {\tilde g}_t$. Then $\pi : (M, g_t) \rightarrow (N, g_N)$ is a Riemannian submersion with totally geodesic fibers for all $t \in (-\epsilon, \epsilon)$, by Theorem \ref{thRSGFKillingglobal}. 

We note that
\begin{eqnarray*}
g_0( [ \xi , P_{\mH}^0 \pi^* W_j  ] , E_1 ) = (\mathcal{L}_{\xi} g_0)(P_{\mH}^0 \pi^* W_j , E_1) ,
\end{eqnarray*}
since $\xi( g_0 ( P_{\mH}^0 \pi^* W_j  , E_1 ) )=0$ by $E_1$ being vertical and
$g_0([\xi,E_1] , P_{\mH}^0 \pi^* W_j )=0$ by the vertical distribution being integrable.

From \eqref{dtsecXURSGF} we obtain on $\mF_y$:
\begin{eqnarray} \label{dtsecXURSGFex}
&& \dt \sec_M (P_{\mH}^t \pi^* X , E_1) \vert_{t=0} \nonumber \\ &&= - \frac{1}{2} \sum\nolimits_{j=n+2}^{n+p} g_0( [P_{\mH}^0 \pi^* W_{n+1}, P_{\mH}^0 \pi^* W_j], E_1 ) \cdot g_0( [V_{n+1} , P_{\mH}^0 \pi^* W_j  ] , E_1 ) \nonumber \\
&&= - \frac{1}{2} \sum\nolimits_{j=n+2}^{n+p} g_0( [ \pi^* W_{n+1},  \pi^* W_j], E_1 ) \cdot g_0( [ (f \circ \pi) \cdot \xi ,  \pi^* W_j  ] , E_1 ) \nonumber \\
&&= - \frac{1}{2} (f \circ \pi) \cdot \sum\nolimits_{j=n+2}^{n+p} g_0( [ \pi^* W_{n+1},  \pi^* W_j], E_1 ) \cdot g_0( [ \xi ,  \pi^* W_j  ] , E_1 ) \nonumber \\
&& \quad + \frac{1}{2} \sum\nolimits_{j=n+2}^{n+p} g_0( [ \pi^* W_{n+1},  \pi^* W_{j}], E_1 ) \cdot g_0( \xi ,  E_1 ) \cdot ( \pi^* W_{j} ) (f \circ \pi) \nonumber \\
&& = - \frac{1}{2} (f \circ \pi) \sum\nolimits_{j=n+2}^{n+p} g_0( [ \pi^* W_{n+1},  \pi^* W_j], E_1 ) \cdot (\mathcal{L}_{\xi} g_0)( \pi^* W_j , E_1) \nonumber \\
&& \quad + \frac{1}{2} g_0( [ \pi^* X,  \pi^* Y], E_1 ) \cdot g_0( \xi ,  E_1 ) \cdot Y (f ) .
\end{eqnarray}
With the assumptions about $\xi$, by adjusting $Y(f)$, we can make the right-hand side of
\eqref{dtsecXURSGFex} 
not constant on $\mF_y$.
Then there exists $s \in (-\epsilon, \epsilon)$ for which $\sec_M (P_{\mH}^s \pi^* X,E_1)$ is not constant on $\mF_y$.
\end{proof}

\begin{example} \label{exampleHopf}
The assumptions of Corollary \ref{corpressecUXex} are satisfied e.g., for  
Hopf fibrations $\pi : (S^7, g_0) \rightarrow (S^4, g_N)$ and
$\pi_m : (S^{4m+3}, g_0) \rightarrow (\mathbb{H}P^m, g_N)$, $m>1$, with $g_0$ being the round metric on the sphere. Then there exist vertical Killing vector fields $\{ E_1, E_2, E_3 \}$ on $(M,g_0)$.
We can set $\xi$ in the proof of Corollary \ref{corpressecUXex} to be one of the fields $\{ E_1, E_2, E_3 \}$, since for all $a \in \{1,2,3\}$, $y \in N$ and $X \in T_y N$ there exists $Y \in T_y N$ such that the function $g_0( [ \pi^* X, \pi^* Y ] ,E_a )$ is not constant on $\mF_y$ \cite{Gray}.
We note that on every fiber $\mF_x \equiv S^3$ there exists also another family of Killing fields $\{ \xi_1, \xi_2, \xi_3 \}$ such that functions $g(E_a, \xi_b)$ are not constant on $\mF_x$ (since the linear space of Killing vector fields on $S^3$ has dimension $6$, see e.g., \cite[Section 2]{Bizon} for an explicit presentation of those Killing fields and relations between them). In fact, vector fields $\xi_a$ on $\mF_x$ can be obtained as $P_{\mV}^0 [ \pi^* X, \pi^* Y ]$ for some $X,Y \in T_x N$ \cite{GromollWalschap}.
\end{example} 

Corollary \ref{corpressecUXex} indicates existence of fat bundles with vertizontal curvatures not constant along fibers, answering negatively Problem 1 in \cite{ZillerFatness}.

We have also the following result about the stability of vertical Killing vector fields with respect to variations preserving a Riemannian submersion and metric induced on its fibers.

\begin{theorem}
Let $\{ g_t, t \in (-\epsilon, \epsilon) \} \subset \Riem(M, \mV, g_0)$ be a 
variation such that $\pi : (M, g_t) \rightarrow (N, g_N)$ is a Riemannian submersion for all $t \in (-\epsilon, \epsilon)$ and \eqref{bsharpinv} holds. Let $K$ be a vertical Killing vector field on $(M,g_0)$, then $K$ is a Killing vector field for all $\{ g_t, t \in (-\epsilon, \epsilon) \}$ if and only if for all $\{ V_{n+1} , \ldots , V_{n+p} \}$ in \eqref{bsharpinv} we have $[V_i , K] =0$.
\end{theorem}
\begin{proof}

Since $K$ is a Killing vector field,  
$\{ g_t, t \in (-\epsilon, \epsilon) \} \subset \Riem(M, \mV, g_0)$ and
$\pi : (M, g_t) \rightarrow (N, g_N)$ is a Riemannian submersion for all $t \in (-\epsilon, \epsilon)$, we have for all $t \in (-\epsilon, \epsilon)$, all vertical fields $U_1, U_2$ and all $g_t$-horizontal fields $Z_1, Z_2$:
\[
0 = (\mathcal{L}_K g_t)(U_1, U_2) = (\mathcal{L}_K g_t)(Z_1, Z_2).
\]

Let $i \in \{n+1, \ldots , n+p\}$ and let $U$ be a vertical vector field. Then, using \eqref{bsharpinv}, integrability of $\mathcal{V}$,  \eqref{BVVzero} and \eqref{dtPHWi} we obtain
\begin{eqnarray*}
&& \dt ( (\mathcal{L}_K g_t)(P_{\mH}^t \pi^* W_i , U) ) =
\dt \big(  K (g_t( P_{\mH}^t \pi^* W_i , U )) - g_t([K , P_{\mH}^t \pi^* W_i ] , U) \\
&& - g_t([K,U] , P_{\mH}^t \pi^* W_i )  \big) \\
&& = - B_t ([K , P_{\mH}^t \pi^* W_i ] , U) - g_t ([K , \dt P_{\mH}^t \pi^* W_i ] , U) \\
&&= g_t ([K , V_i ] , U).
\end{eqnarray*}
\end{proof}

In Example \ref{exampleHopf}, if we take $\xi \in \{\xi_1, \xi_2, \xi_3\}$ then $E_1, E_2, E_3$ will remain Killing fields on $(S^{4m+3}, g_t)$, since $[\xi_a, E_b] =0$ for all $a,b \in \{1,2,3\}$. On the other hand, for all $a \in \{1,2,3\}$ and all $x \in N$ the restriction of $\xi_a$ to a fiber $\mF_x$ will be a Killing field on $(\mF_x, g_0 \vert_{\mF_x})$, but in general $\xi_a$ may not be a Killing field on $(S^{4m+3}, g_t)$ for $t \neq 0$.

The metric properties of $(M,g_s)$ obtained in Corollary \ref{corpressecUXex} 
may be significantly different than those of $(M,g_0)$. 
In particular, metrics obtained in Example \ref{exampleHopf} will not be $3$-Sasakian. However, we note that manifolds in Example \ref{exampleHopf} are $3$-sphere bundles, which over $S^4$ are classified up to isomorphism \cite{DerdzinskiRigas}.

\section*{Acknowledgements}
The author is grateful to the late Professor Pawe\l \ Walczak for helpful remarks about an early version of this paper.


\begin{thebibliography}{99}

\bibitem{AlvarezLopez}
Alvarez Lopez, J.A.: The basic component of the mean curvature of Riemannian foliations,
Ann. Glob. Anal. Geom. 10: 179-194 (1992)

\bibitem{BEM} 
Bourguignon, J.-P., Ebin, D.G., Marsden, J.E.: Sur le noyau des operateurs pseudodifferentiels a symbole surjectif et non injectif, C. R. Acad. Sci. Paris Ser. A 282 , 867-870 (1976)

\bibitem{Besse}
Besse, A. L.: {Einstein manifolds}. Springer-Verlag, Berlin (1987)

\bibitem{Bizon}  
Beig R.,  Bizo\'n P.,  Simon W.: {Vacuum initial data on S3 from Killing vectors}, Class. Quantum Grav. 36 215017

\bibitem{Cheeger}
Cheeger, J.: Some examples of manifolds of nonnegative curvature. J. Diff. Geom. 8, 623-628 (1973)

\bibitem{Delay}
Delay, E.: Smooth compactly supported solutions of some underdetermined elliptic PDE, with gluing applications. Communications in Partial Differential Equations, 37(10), 1689-1716 (2012) 

\bibitem{DerdzinskiRigas} 
Derdzinski, A., Rigas, A.: Unflat connections in 3-sphere bundles over S4,
Trans. A.M.S. 265, 485-493 (1981)

\bibitem{Ehresmann}
Ehresmann, C.: 
Les connexions infinit\'esimales dans un espace fibr\'e diff\'erentiable.
Colloque de topologie (espaces fibr\'es), Bruxelles, 1950, 
29-55.
Georges Thone, Li\`ege; Masson \& Cie, Paris (1951)

\bibitem{Gray}
Gray, A.: {Pseudo-Riemannian almost product manifolds and submersions}. J. Math. Mech., 16, no. 7, 715-737 (1967)
%

\bibitem{GromollWalschap}
Gromoll, D., Walschap, G.: {Metric foliations and curvature}, Birkh\"{a}user (2009)

\bibitem{KobayashiNomizu} Kobayashi, S., Nomizu, N.: 
{Foundations of differential geometry I}, Interscience, New York (1963)

\bibitem{Lang}
Lang, S.: Fundamentals of differential geometry. Springer-Verlag, New York (1999)

\bibitem{March}
March, P., Min-Oo, M., Ruh, E.A.: Mean curvature of Riemannian foliations, Canad. Math. Bull. Vol. 39 (1),  
95-105 (1996)

\bibitem{Nagy}
Nagy, P.T.: On bundle-like conform deformation of a Riemannian submersion. Acta Mathematica Academiae Scientiarum Hungaricae 39, 155-161 (1982)

\bibitem{O'Neill}
O'Neill, B.: The fundamental equations of a submersion. Mich. Math. J. 13, 459-469 (1966)

\bibitem{rz-2}
Rovenski V., Zawadzki, T.: Variations of the total mixed scalar curvature of a distribution,
Ann. Glob. Anal. Geom. 54 , 87-122 (2018)

\bibitem{Tondeur} 
Tondeur,  P.: {Geometry of foliations}, Springer-Verlag (1997)

\bibitem{TZCMS}
Zawadzki, T.: {A variational characterization of contact metric structures}, Ann. Glob. Anal. Geom. 62, 129-166 (2022)

\bibitem{Ziller}
Ziller, W.: Examples of Riemannian manifolds with nonnegative sectional curvature, in: Metric and Comparison Geometry, Surv. Diff. Geom. 11, ed. K. Grove and J. Cheeger, 63-102 (2007)

\bibitem{ZillerFatness}
Ziller, W.: {Fatness Revisited}, lecture notes (2005).

\end{thebibliography}
\end{document}